\documentclass[12pt]{siamart190516}

\title{Accurate  error estimation in CG\thanks{Version of \today. The work of J.~Pape\v{z} and P.~Tich\'{y} was supported by the Grant Agency of the Czech Republic under grant no. 20-01074S}}

\author{G\'erard Meurant\footnotemark[1],
        Jan Pape\v{z}\footnotemark[2]
        \and Petr Tich\'{y}\footnotemark[3] }

\headers{Accurate error estimation in CG}{G.~Meurant, J.~Pape\v{z}, and P.~Tich\'{y}}

\usepackage{amsmath}
\usepackage{amssymb}
\usepackage{graphicx}
\usepackage{epstopdf}
\usepackage{color}
\usepackage{algorithmic}
\usepackage{listings}
\usepackage{array}
\usepackage{multirow,bigdelim}
\usepackage{cleveref}

\definecolor{mygreen}{RGB}{28,172,0} % color values Red, Green, Blue
\definecolor{mylilas}{RGB}{170,55,241}

\newcommand{\klein}[1]{\scriptscriptstyle #1}

\newcommand{\cfa}{\alpha}        % the first CG coefficient
\newcommand{\cfb}{\beta}         % the second CG coefficient
\newcommand{\cfgq}{\Delta}         % the second CG coefficient ????????
\newcommand{\smu}{\scriptscriptstyle (\mu)}

\newcommand{\MT}[1]{\widetilde{\Delta}_{#1}}
\newcommand{\AN}[1]{\varepsilon_{#1}}

\ifpdf
\hypersetup{
  pdftitle={Accurate error estimation in CG},
  pdfauthor={G\'erard Meurant, Jan Pape{\v{z}}, and Petr Tich\'{y}}
}
\fi

\begin{document}

\maketitle

\renewcommand{\thefootnote}{\fnsymbol{footnote}}
\footnotetext[1]{Paris, France.
E-mail: {\tt gerard.meurant@gmail.com}.}
\footnotetext[2]{Institute of Mathematics, Czech Academy of Sciences, Prague, Czech Republic. E-mail: {\tt papez@math.cas.cz}.}
\footnotetext[3]{Faculty of Mathematics and Physics, Charles University, Prague, Czech Republic. E-mail: {\tt petr.tichy@mff.cuni.cz}.}

\begin{abstract}
In practical computations, the (preconditioned) conjugate gradient (P)CG method is the iterative method of choice for solving systems of linear algebraic equations $Ax=b$ with a real symmetric positive definite matrix $A$.
During the iterations it is important to monitor the quality
of the approximate solution $x_k$ so that
the process could be stopped whenever $x_k$
is accurate enough. One of the most relevant quantities
for monitoring the quality of $x_k$ is the squared $A$-norm of the error vector $x-x_k$. This quantity cannot be easily evaluated, however, it can be estimated. Many of the existing estimation techniques are inspired by
the view of CG as a procedure for approximating a certain Riemann--Stieltjes integral.
The most natural technique is based on the Gauss quadrature approximation and
provides a lower bound on the quantity of interest. The bound can be cheaply evaluated using terms that have to be computed anyway in the forthcoming CG iterations. If the squared $A$-norm of the error vector decreases rapidly, then the lower bound represents a tight estimate. In this paper we suggest  a heuristic strategy
aiming to answer the question of how many forthcoming CG iterations are needed to get an estimate with the prescribed accuracy. Numerical experiments demonstrate that the suggested strategy is efficient and robust.

\end{abstract}



\section{Introduction}
 \label{sec:introduction}

Nowadays, the (preconditioned) conjugate gradient (P)CG method of Hestenes
and Stiefel \cite{HeSt1952} is the method of choice for solving large
and sparse systems of linear algebraic equations
\begin{equation}
Ax=b\label{eqn:Axb}
\end{equation}
with a real symmetric positive definite matrix $A$ of order $n$.
In theory, the CG method has the best of the two worlds of iterative
and direct linear solvers. As an iterative method, it has low memory
requirements and provides an approximate solution $x_{k}$ at each
iteration allowing to stop the algorithm whenever a stopping criterion
is met. As a direct method, it finds (assuming exact arithmetic) the
solution after at most $n$ iterations, and the approximate
solution $x_{k}$ is optimal in the sense that it minimizes the squared
$A$-norm (also called the energy norm) of the \emph{error vector}
$x-x_{k}$,
\begin{equation}
\varepsilon_{k}\equiv\Vert x-x_{k}\Vert_{A}^{2}=(A(x-x_{k}),x-x_{k}),\label{eq:error}
\end{equation}
over the underlying linear manifold. Note that $\left\Vert x-y\right\Vert _{A}^{2}$ as a function of $y$ was called ``the error function'' in the original
Hestenes and Stiefel paper \cite{HeSt1952}. The authors mention
\cite[p.~413]{HeSt1952} that this function can be used as a measure
of the ``goodness'' of $x_{k}$ as an approximation of $x$. Following
them and for the sake of simplicity, we will call $\varepsilon_{k}$
the \emph{error}. In the context of solving various real word problems, the error $\varepsilon_{k}$ has an important meaning, e.g.,
in physics, quantum chemistry, or mechanics, and plays a fundamental role in evaluating convergence and estimating the algebraic error in the context of numerical solving of PDE's \cite{Ar2004,JiStVo2010,ArGeLo2013,DoTi2020}.

The theoretical properties of CG can be described in many mathematically
equivalent ways. For example, the connection with orthogonal polynomials,
which has already been noticed in \cite{HeSt1952}, allows seeing
CG as a procedure for computing Gauss quadrature approximation to
a Riemann--Stieltjes integral. This view of CG is very useful not
only for estimating norms of the error vectors, but also for analyzing
the CG behavior in finite precision arithmetic. During many finite precision
CG computations, the orthogonality among the residual vectors is usually lost
quickly and convergence is delayed \cite{Gr1989,GrSt1992}. Based
on previous results of Paige \cite{Pa1980a}, it has been shown by
Greenbaum~\cite{Gr1989} that the results of finite precision CG computations
can be interpreted (up to some small inaccuracy) as results of exact
CG applied to a larger system with a system matrix having many eigenvalues
distributed throughout tiny intervals around the eigenvalues of $A$.
In other words, the results of \cite{Gr1989} show that some theoretical
properties of CG remain valid even in finite precision arithmetic.
For a summary of the properties of the CG and Lanczos algorithms in finite precision arithmetic see, e.g., \cite{B:Me2006,MeSt2006}.

Inspired by the connection of CG with Riemann--Stieltjes integrals,
a way of research on estimating norms of the error vector was started
by Gene Golub in the 1970s and continued throughout the years with
several collaborators (e.g., G.~Dahlquist, S.~Eisenstat, S.~Nash,
B.~Fischer, G.~Meurant, Z.~Strako\v{s}). The main idea is to approximate
the Riemann--Stieltjes integral of a suitable function by a Gauss or another modified quadrature rule (Gauss--Radau, Gauss--Lobatto, anti-Gauss etc.), and
try to bound the estimated quantity \cite{GoMe1994,GoSt1994,GoMe1997}. Note that the resulting
bounds can sometimes be very inaccurate approximations to the quantity of interest.

As already mentioned above, the error $\varepsilon_{k}$ must play a prominent role
in evaluating CG convergence. In this paper we concentrate on its accurate
estimation. Following the ideas of \cite{GoSt1994,GoMe1997,StTi2002,MeTi2013,MeTi2019},
one can improve the accuracy of the bounds, considering quadrature
rules at iterations $k$ and $k+d$ for some integer $d\geq0$ called
the \emph{delay}. At CG iteration $k+d$, an improved estimate of
the error at iteration $k$ is obtained. For a detailed construction
of the (lower and upper) bounds on $\varepsilon_{k}$ based on the
delay approach see, e.g., \cite[Section 2.5]{MeTi2013}. The larger
the delay is, the better are the bounds at iteration $k$. However,
a constant value of $d$ is usually not sufficient in the initial
stage of convergence, and it may require too many extra steps of CG in
the convergence phase. Hence, there is a need for developing a heuristic
technique to choose $d$ adaptively at each iteration, to reflect
the required accuracy of the estimate.

In this paper, we focus on estimating the error $\varepsilon_{k}$
using the Gauss quadrature approach that provides a lower bound. We address
the problem of the adaptive choice of $d$. In particular, we develop
an adaptive (heuristic) strategy, which aims to ensure that the lower
bound approximates the error $\varepsilon_{k}$ with a prescribed
tolerance. Note that an analogous technique can be also used to improve
the accuracy of other bounds based, e.g., on Gauss--Radau quadrature.

The paper is organized as follows.
In \Cref{sec:CG} we introduce the notation, recall the standard Hestenes and Stiefel version of the CG algorithm, and present formulas that will be used
for constructing the lower bound of the error.
In \Cref{sec:adaptd}, we describe an adaptive strategy for choosing the delay~$d$, aiming to meet the prescribed accuracy of the lower bound.
\Cref{sec:PCG} shows how to modify the formulas for preconditioned CG.
%when using preconditioning.
%To make the adaptive strategy more reliable in hard cases, some possible extensions are presented in \Cref{sec:extensions}.
Some possible extensions to make the adaptive strategy more reliable in hard cases are presented in \Cref{sec:extensions}.
Results of numerical experiments are given in \Cref{sec:experiments} and the paper ends with a concluding discussion. In \Cref{sec:MATLABcode} we provide a simplified MATLAB code of the suggested algorithm.

\section{The CG algorithm and the lower bound on the error}
 \label{sec:CG}
%{\em The aim of this section is to introduce notation, present basic formulas for estimating the $A$-norm of the error, lower bounds, upper bounds, delay $d$. We will also ontroduce the new upper bound based on $\kappa(A)$ and compare it with the existing ones. We would like to control the accuracy of the estimates. }

%In this section, we first present the standard CG algorithm. Then, we list several properties of errors and quantities in CG computations, which will be used in our developments below. Finally, a lower bound on the $A$-norm of the error~$\AN{k}$ is presented.

\begin{algorithm}[ht]
\caption{Conjugate Gradients} \label{alg:cg}

\begin{algorithmic}[0]
\STATE \textbf{input} $A$, $b$, $x_{0}$
\STATE $r_{0}=b-Ax_{0}$, $p_{0}=r_{0}$
%, $\cfa_{0}=\frac{r_{0}^{T}r_{0}}{p_{0}^{T}Ap_{0}}$
\FOR{$k=0,\dots$ until convergence}
    \medskip
    \STATE $\cfa_{k}=\frac{r_{k}^{T}r_{k}}{p_{k}^{T}Ap_{k}}$
    \STATE $x_{k+1}=x_{k}+\cfa_{k}p_{k}$
    \STATE $r_{k+1}=r_{k}-\cfa_{k}Ap_{k}$
         \hspace*{15mm}%
        \smash{$\left.\begin{array}{@{}c@{}}\\{}\\{}\\{}\\{}\\{}\end{array}
        \right\} \begin{tabular}{l}{\tt cgiter(k)}\end{tabular}$}
    \STATE $\cfb_{k+1}=\frac{r_{k+1}^{T}r_{k+1}}{r_{k}^{T}r_{k}}$
    \STATE $p_{k+1}=r_{k+1}+\cfb_{k+1}p_{k}$
%    \STATE $\cfa_{k}=\frac{r_{k}^{T}r_{k}}{p_{k}^{T}Ap_{k}}$
    \medskip
\ENDFOR
\end{algorithmic}

\end{algorithm}

The classical version of the Conjugate Gradient method (Hestenes and Stiefel, \cite{HeSt1952}) is given by \cref{alg:cg}. For later use, we denote the function that  performs one CG iteration for updating the approximations~$x_k$, corresponding residuals~$r_k$, direction vectors~$p_k$, and computing the CG coefficients~$\cfa_{k}$ and~$\cfb_{k+1}$ by {\tt cgiter}(k). %Moreover, the quantity
%\[
%\pi_k \;\equiv\; \frac{\|r_k\|^2}{\| p_k \|^2}
%\]
%is updated in an efficient and numerically stable way; see \cite[p.~943]{MeTi2019}.
% ---------------------------------------------------------------------
%\begin{algorithm}[ht]
%{\bf function} $\left[x_{k}, r_{k}, p_{k}, \cfa_{k}, \cfb_{k}\right] =$ {\tt cgiter}($A,x_{k-1}, r_{k-1}, p_{k-1}, \cfa_{k-1}$)\\
%%
%\begin{algorithmic}[0]
%\STATE \textbf{init} $p_0 = r_{0}=b-Ax_{0}$, $\cfa_{0}=\dfrac{r_{0}^{T}r_{0}}{p_{0}^{T}Ap_{0}}$%,\  $\pi_{0}=1$
%\STATE {\parbox{6cm}{
%\begin{eqnarray*}
%x_{k}&=&x_{k-1}+\cfa_{k-1}p_{k-1}\\
% r_{k}&=&r_{k-1}-\cfa_{k-1}Ap_{k-1}\\
% \cfb_{k}&=&\frac{r_{k}^{T}r_{k}}{r_{k-1}^{T}r_{k-1}}\\
% p_{k}&=&r_{k}+\cfb_{k}p_{k-1}\\
% \cfa_{k}&=&\frac{r_{k}^{T}r_{k}}{p_{k}^{T}Ap_{k}}%,\quad
% %\pi_k = \frac{\pi_{k-1}}{\pi_{k-1}+\cfb_k}
%\end{eqnarray*}}}
%\end{algorithmic}
%\end{algorithm}
% ---------------------------------------------------------------------

The theoretical properties of CG are well known and we do not list them here; see, for instance, \cite{B:Me2006}. In our context it is only important that the sequence of errors $\{\AN{k}\}$ is decreasing and that the value of the difference between two consecutive terms is known. The proof of results of the following lemma can be found in standard literature; see, e.g., \cite{GoSt1994, B:Me2006, StTi2002}.

%The errors and computed quantities in CG algorithm feature several important properties. We list some of them in the lemma below. The proofs can be found in a standard literature, we refer, for example, to~\cite{GoSt1994}.

\begin{lemma}
 \label{lem:basic}
Until the solution is found, i.e.,~$x_{s}=x$, it holds that
\begin{equation}
 \label{eqn:localdecrease}
 \AN{k} = \cfa_k \| r_k \|^2 + \AN{k+1}, \qquad k = 0, 1, \ldots, s-1,
\end{equation}
with $\cfa_k \| r_k \|^2 > 0$, $k = 0, 1, \ldots, s-1$.
In particular, the squared $A$-norm of the error vector $x-x_k$ in CG is decreasing, i.e.,
\begin{equation}
 \label{eqn:monoCG}
 \AN{k} > \AN{k+1}, \qquad k = 0, 1, \ldots, s-1.
\end{equation}
\end{lemma}

Note that \cref{lem:basic} assumes exact arithmetic so that
the algorithm always finds the exact solution, i.e., $x_s=x$ for some $s\leq n$.
In finite precision computations, the errors $\AN{k}$ usually do not approach zero, but reach an ultimate level
 of accuracy proportional to the squared machine precision,
and then, they stagnate. Nevertheless, it has been shown in \cite{StTi2002} that until the ultimate level of accuracy is reached,
the identity \eqref{eqn:localdecrease} and the inequality \eqref{eqn:monoCG} hold also for the computed quantities, up to some small inaccuracy. Hence, considering iterations $k$ before this level is reached, we can assume
that \eqref{eqn:localdecrease} and \eqref{eqn:monoCG} do hold also
during finite precision computations.

The relation \eqref{eqn:localdecrease} is the basis for the derivation of the lower bound on the error. Given a nonnegative integer $d$
and assuming that \eqref{eqn:localdecrease} holds for iterations
$k$, $k+1$, $\dots$, $k+d$, with $k+d\le s-1$, we obtain
\begin{equation}
 \label{eqn:dlocaldecrease}
 \AN{k} \;=\; \sum^{k+d}_{j=k} \cfa_{j} \| r_{j} \|^2 \;+\; \AN{k+d+1}\,,%\underbrace{\sum^{k+d-1}_{j=k} \cfa_{j} \| r_{j} \|^2}_{\cfgq_{k:k+d-1}} \;+\; \AN{k+d}\,,
\end{equation}
see \cite{StTi2002}, and we define
\begin{equation}
 \label{eqn:Delta}
 \cfgq_{k:k+d} \equiv \sum^{k+d}_{j=k} \cfa_{j} \| r_{j} \|^2.
\end{equation}
Note that instead of $\cfgq_{k:k}$ we write $\cfgq_{k}$.
Using the fact that $\AN{k+d+1} \ge 0$, the quantity \eqref{eqn:Delta}
%
%\begin{equation}
%\label{eqn:LB}
%	\cfgq_{k,d} \equiv \sum^{k+d-1}_{j=k} \cfa_{j} \| r_{j} \|^2 < \AN{k}
%\end{equation}
represents a lower bound on $\AN{k}$.
The indisputable advantage of the bound \eqref{eqn:Delta} is that it is very cheap to compute since it only involves quantities which have to be computed in CG, and,
as already mentioned, it works under natural assumptions also for finite precision computations; for more details see~\cite{StTi2002}.

%This bound has several favourable features:
%\begin{itemize}
%	\item The quantities for evaluating~$\cfgq_{k,d}$ are at our disposal \cb{during the} CG computations. Its evaluation cost is therefore negligible.
%	\item As thoroughly examined and proved in~\cite{StTi2002}, this bound is numerically stable and is guaranteed (up to a sufficiently small inaccuracy) also in the finite-precision computations where many theoretical CG properties do not hold.
%\end{itemize}

While \eqref{eqn:dlocaldecrease} suggests that the computable lower bound
\eqref{eqn:Delta} is tight for sufficiently large $d$, the choice of $d$
that ensures a sufficient accuracy of the lower bound in practical computations,
is usually unknown. A constant value of $d$ can be fine in some
special cases with a rapid and linear convergence. However, in general, we expect that the convergence of the $A$-norm of the error vector $x-x_k$ in (P)CG is irregular. Quasi-stagnation can alternate with convergence periods, and in a convergence period we can often observe  linear, superlinear, but also sublinear convergence.

To illustrate numerically the behavior of the bound $\cfgq_{k:k+d}^{\klein{1/2}}$, we consider the matrix {\tt s3dkq4m2} of order $n = 90449$ that  can be downloaded from the
SuiteSparse Matrix collection\footnote{https://sparse.tamu.edu}.
%the
%CYLSHELL collection in the Matrix Market library.
The right-hand side vector $b$ is randomly generated and normalized to~$1$. The factor $L$ in the preconditioner
$M = LL^T$ is determined by the MATLAB incomplete Cholesky (ichol) factorization with threshold dropping,
{\tt type = 'ict'}, {\tt droptol = 1e-5}, and with the global diagonal shift {\tt diagcomp = 1e-2}.
\begin{figure}[!htbp]
\centering
 \includegraphics[width=8cm]{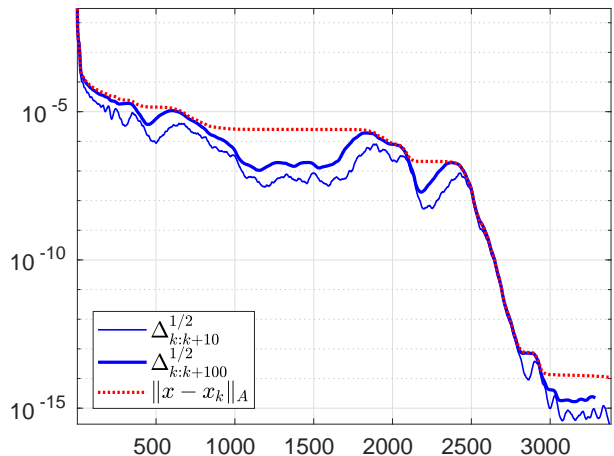}
\caption{Matrix {\tt s3dkq4m2}: the $A$-norm of the error vector $x-x_k$ (dotted curve) and the  lower bound $\Delta_{k:k+d}^{\klein{1/2}}$  for two fixed values of~$d$: $d=10$ (solid curve) and $d=100$ (thick solid curve).} \label{fig:01}
\end{figure}

In \Cref{fig:01} we can observe that $\| x-x_k \|_A$ (dotted curve) exhibits an irregular behaviour. The lower bound $\Delta_{k:k+d}^{\klein{1/2}}$  for $d=10$ (solid curve) underestimates
significantly the quantity of interest
%by several orders of magnitude
in the case of quasi-stagnation, but,
when convergence starts around iteration 2500, the lower bound
$\Delta_{k:k+d}^{\klein{1/2}}$ is sufficiently accurate.
The choice $d=100$ (thick solid curve) improves the results, but $\| x-x_k \|_A$ is still underestimated by several orders of magnitude. However,
in the final convergence phase we now use a value of $d$ larger than necessary, and, therefore, we waste computational resources, particularly, if the error is used as a criterion to stop the iterations. One may think that we do not have to care too much about the quasi-stagnation phase, but if the underestimation of the $A$-norm of the error is too large, we may undesirably stop the iterations too soon when we are far of having a good approximation of the solution.

This example demonstrates clearly a need for developing a technique to choose $d$ adaptively.
%to use
%\eqref{eqn:dlocaldecrease}
%\eqref{eqn:Delta} in practical computations
%to improve the accuracy
%of the lower bounds, we should be able to choose $d$ adaptively.
As we said, if our stopping criterion is based on the lower bound \eqref{eqn:Delta} and $d$ is too small, we are at risk to stop the iterations too early.
%Otherwise we are at risk that we stop \cred{the iterations} too early \cred{{\bf ??? late???}}, if our stopping criterion is based on the lower bound \eqref{eqn:Delta}.
At the same time, we would like to keep~$d$ as small as possible to avoid unnecessary iterations.
This is a nontrivial task. As we will see later, we can never guarantee a completely safe choice of $d$. Nevertheless, we are able to suggest a strategy that works satisfactorily in most cases.

\section{The adaptive choice of~$d$}
 \label{sec:adaptd}
In this section we set a requirement on the accuracy of the bound, introduce the ideal value of~$d$, which ensures this accuracy, and present a first adaptive strategy for the choice of~$d$.

%The section ends with a pseudo-algorithm for the adaptive choice of~$d$.

\subsection{Prescribing the accuracy of the estimate}

Given a prescribed tolerance $\tau \in (0,1)$, we require that the relative error of the lower bound satisfies
\begin{equation}
 \label{eqn:tolerance}
 \frac{\AN{k}-\cfgq_{k:k+d}}{\AN{k}} \leq \tau.
\end{equation}
By simple manipulations, this is equivalent to
\begin{equation}
 \label{eqn:aUP}
 \varepsilon_{k} \leq \frac{\cfgq_{k:k+d}}{1-\tau}.
\end{equation}
In other words, if \eqref{eqn:tolerance} holds, then we also get an upper bound on the error. We observe that the relative error of the upper bound \eqref{eqn:aUP} is bounded above by $\tau/(1-\tau)$. Note also that if \eqref{eqn:tolerance} holds, then
\begin{equation}
 \label{eqn:toleran2}
 \frac{\Vert x - x_k \Vert_A - \cfgq_{k:k+d}^{1/2}}{\Vert x - x_k \Vert_A} < \tau,
\end{equation}
i.e., the $A$-norm $\Vert x - x_k \Vert_A=\varepsilon_{k}^{1/2}$ is approximated by $\cfgq_{k:k+d}^{1/2}$ with a relative accuracy less than $\tau$.

Using \eqref{eqn:dlocaldecrease}, one can rewrite \eqref{eqn:tolerance} as
\begin{equation}
 \label{eqn:tolerance2}
 \frac{\varepsilon_{k+d+1}}{\varepsilon_{k}} \leq \tau.
\end{equation}
It means that $d$ should be ideally chosen such that the error decreases sufficiently in $d+1$ iterations.
We now translate this requirement into the problem of choosing a proper value of~$d$.
%
%\cred{What is with parenthesis appears as a repetition. Remove it?}
%
\begin{definition}%[Ideal value of~$d$]
 \label{def:ideal}
	We define the \emph{ideal value} of~$d$ at a given iteration~$k$ to be the minimal value of $d$ such that \eqref{eqn:tolerance2} holds, and denote it by $\widetilde{d}_k$. To simplify the notation we usually omit the subscript $k$.
\end{definition}

As demonstrated in \Cref{fig:01}, a constant value of $d$ will usually not be satisfactory,
where ``satisfactory'' means that it satisfies \eqref{eqn:tolerance2} at each iteration and avoids $\widetilde{d} \ll d$.
In other words, the value of $d$ should depend on the iteration number~$k$,
$d=d_k$ (again to simplify the notation we will omit the subscript $k$).
%Analogously, for the ideal value, $\widetilde{d} = \widetilde{d}(k)$, $k = 0, 1, \ldots$.
%
The adaptive choice of $d$ that reflects the required accuracy and tightly approximates $\widetilde{d}$ is a challenging problem we tackle in the rest of the section.

\subsection{An adaptive strategy for choosing~$d$}
 \label{sec:adaptd_later}

The strategy we propose is based on replacing the unknown errors in \eqref{eqn:tolerance2} by their estimates. For the denominator, we use the lower bound \eqref{eqn:Delta}, which is available at the iteration~$k+d$.
The approximation of the numerator using the available information is a complicated task. If some a~priori information is known, in particular a tight underestimate of the smallest
eigenvalue of~$A$, one may try to use the upper bound based on Gauss--Radau quadrature \cite{MeTi2019}.
However, as mentioned in \cite{MeTi2019} and discussed in more detail in Section~\ref{sec:adaptd_init}, even if we know the smallest eigenvalue to high accuracy, the Gauss--Radau upper bound is usually delayed in later iterations, and, therefore,
does not always provide a sufficiently accurate information on the current error.

%For the nominator, a natural idea is to replace~$\varepsilon_{k+d}$ by an upper bound, however, a tight bound is, in general, not available.
Here our main idea is to bound the error~$\varepsilon_{k+d+1}$ using the available (poor) lower bound~$\cfgq_{k+d+1}$ that can be computed in the iteration~$k+d+1$, and a \emph{safety factor}~\mbox{$S>1$} (that again depends on~$k$) such that
\begin{equation}
 \label{eqn:H}
  \varepsilon_{k+d+1} \leq S  \cfgq_{k+d+1}.
\end{equation}
%where the value of $S>1$ is to be properly set as we will see below.
%can be tuned based on the \cb{lower bounds in the previous iterations. }%reconstructed convergence history.
Having some heuristic value of $S$ that would ideally satisfy the
inequality \eqref{eqn:H} in hands at iteration~$k$, we choose $d$ as the minimal value satisfying
\begin{equation}
 \label{eqn:lowerb_strategy}
 \frac{S \cfgq_{k+d+1}}{\cfgq_{k:k+d}} \leq \tau.
\end{equation}
The following lemma shows that $S$ can be bounded from above by  $\kappa(A)$.

\begin{lemma}
 \label{lem:kappa}
%Assuming exact arithmetic and
In the notation introduced above, it holds that
$
    \varepsilon_{k} \leq \kappa(A) \cfgq_{k}.
$
\end{lemma}
\begin{proof}
To get the lower bound on $\cfgq_{k} = \cfa_{k}\|r\|_{k}^{2}$, we use the inequality
\[
p_{k}^{T}Ap_{k}\leq r_{k}^{T}Ar_{k}
\]
which can be proved using only local orthogonality, that is, orthogonality of two consecutive vectors; see \cite[Lemma~2.31]{B:Me2006}. Denoting $\lambda_{\min}$ and $\lambda_{\max}$  the smallest and the largest eigenvalue of~$A$, it holds that
\[
 {\cfgq_{k}} = \frac{\|r_{k}\|^{2}}{p_{k}^{T}Ap_{k}}\|r_{k}\|^{2}\geq\frac{\|r_{k}\|^{2}}{r_{k}^{T}Ar_{k}}\|r_{k}\|^{2}\geq\frac{1}{\lambda_{\max}}\|r_{k}\|^{2}\geq\frac{\lambda_{\min}}{\lambda_{\max}} \| x - x_k \|_A^2,
\]
where we have used $r_k = b - Ax_k = A(x-x_k)$. Finally,
$
	{\cfgq_{k}} \geq \kappa(A)^{-1}\varepsilon_k.
$
\end{proof}

Note that the local orthogonality used in the proof of \Cref{lem:kappa} is well preserved during
finite precision computations; see \cite{StTi2002}.
Also, the recursively computed residual $r_k$ corresponds with the true residual $b-Ax_k$ during finite precision CG computations
until the ultimate level of accuracy is reached.
In summary, one can expect that the inequality in \Cref{lem:kappa} holds also during finite precision computations (up to some small unimportant inaccuracy) until the ultimate level of accuracy is reached.

%Here we recall that $\cfgq_{k+d,1} = \cfa_{k+d} \| r_{k+d}\|^2$.

Before introducing the formula for~$S$, we make one important remark.
At iteration $k+d+1$, a new term
$
 \cfgq_{k+d+1}
$
% \cfgq_{k+d+1} = \cfa_{k+d+1}\|r_{k+d+1}\|^2
is available. Recalling relation \eqref{eqn:dlocaldecrease}, this term can
be added to all previous lower bounds on the errors to obtain better approximations. In \cite[Sect.~5.2]{StTi2005}, this technique is called ``reconstruction of the convergence curve''. In particular, the previous lower bounds on the errors $\varepsilon_{\ell}$ can be improved using all the available information by
\[
	\varepsilon_\ell \approx \sum^{{k+d+1}}_{j = \ell} \cfa_j \| r_j \|^2 = \cfgq_{\ell:{k+d+1}}, \qquad {0 \leq \ell \leq k+d},
\]
and $\cfgq_{\ell:{k+d+1}}$ represents again a lower bound on $\varepsilon_\ell$.

Our numerical experiments show that a proper value of $S$ can vary with the iterations. Therefore, our aim is to vary~$S$ based on the information we get from the previous iterations.
%
%The idea behind
The choice of~$S$ should account for the underestimation of~$\varepsilon_{k+d+1}$ by the (simplest) lower bound~$\cfgq_{k+d+1}$. Therefore, at iteration~$k$ we choose it as the largest underestimation in some number of the latest iterations,
\begin{equation}
 \label{eqn:chooseH}
 S \equiv \max_{m \leq \ell \leq k+d} \frac{{\cfgq_{\ell:k+d+1}}}{\cfgq_{\ell}}
  \approx \max_{m \leq \ell \leq k+d} \frac{\varepsilon_{\ell}}{\cfgq_{\ell}}.
\end{equation}
Here~$m$ could be set as 0, meaning that the 
overall convergence history is used. However, then $S$ would be nondecreasing, which does not reflect the observed behavior. Therefore, we use information only from
the latest iterations that caused a significant decrease in the error from iteration $m$ to iteration~$k$, say,
about four orders of magnitude (i.e.~two orders of magnitude in the $A$-norm of the error vector $x-x_k$). In other words, we define  $m$
to be the largest $\ell$, $0\leq \ell <k$ such that
\begin{equation}
 \label{eqn:Happrox}
 \frac{\cfgq_{k:k+d+1}}{\cfgq_{\ell:k+d+1}} \leq \mbox{TOL},\qquad
 \mbox{TOL} \equiv 10^{-4}.
\end{equation}
If the condition \eqref{eqn:Happrox} is not satisfied for any admissible $\ell$, we define $m=0$. In this way, only the latest part of the reconstructed convergence curve is used to determine~$S$.

Note that the computation of the bound using \eqref{eqn:Delta} requires~$d$ additional CG iterations. In practice, however, one runs the CG algorithm, and estimates the error in a backward way, i.e., $d$~iterations back.
The CG algorithm with the adaptive choice of $d$ based on \eqref{eqn:Delta} and \eqref{eqn:chooseH}
is given in~\Cref{alg:pseudo}. When a stopping criterion is satisfied
on line 14 of \Cref{alg:pseudo}, one can use the latest available approximate solution~$x_{\ell}$, thanks to the monotonicity~\eqref{eqn:monoCG} of errors in CG.

%In \Cref{alg:pseudo} we use a switch {\tt init\_{}phase} to distinguish the phase when we just increase $d$ to satisfy the condition \eqref{eqn:initial}; see Section~\ref{sec:adaptd_init}. Once \eqref{eqn:initial} is satisfied, the initial value of $d$ is set, and we can use the strategy of Section~\ref{sec:adaptd_later} for increasing and decreasing the value of~$d$.

% It is not so obvious that Algorithm 3.1 corresponds to what we have defined before since there is no $m$ in the algorithm.

\begin{algorithm}
\caption{{CG with the adaptive choice of $d$}} \label{alg:pseudo}
\algsetup{indent=2em}
\begin{algorithmic}[1]

\STATE \textbf{input} $A$, $b$, $x_{0}$, $\tau$, $\mbox{TOL}$

\STATE $r_{0}=b-Ax_{0}$, $p_{0}=r_{0}$

\STATE {${d=0}$, $k=0$}

\FOR{{$\ell=0,\dots,$}}

\STATE {\tt cgiter($\ell$)} %\quad \textcolor{gray}{\emph{$\ell$th CG iteration}}
\STATE compute~$\cfgq_{\ell}$

\IF{$\ell>0$}
%\STATE $k := \ell-d-1$
\STATE compute~$\cfgq_{k:k+d}$
% from \eqref{eqn:Delta}
\STATE determine~$m$ and $S$ using  \eqref{eqn:Happrox} and \eqref{eqn:chooseH}
\WHILE{$d\geq 0$ and \eqref{eqn:lowerb_strategy}}
% with \jan{$k := \ell-d-1$}

\STATE accept $\cfgq_{k:k+d}$ as an estimate for $\varepsilon_k$	

\STATE $k = k + 1$, ${d=d-1}$

%\IF{$d\geq 0$}
\STATE compute $\Delta_{k:k+d}$, if $d\geq 0$
%\ENDIF
%\STATE update criterion using new $d$
\ENDWHILE

%\IF{$k > 0$}
\STATE use the latest estimate $\Delta_{k-1:k+d}$ in stopping criteria, if $k > 0$
%use $\nu_{k-1}$ in stopping criteria
%\ENDIF
\STATE {${d=d+1}$}

\ENDIF
\ENDFOR

%\STATE \textbf{output} guaranteed lower bounds \jan{$\cfgq_{k:k+d_k}$}, approximate upper bounds \cred{$\cfgq_{k:k+d_k} / (1 - \tau)$}

\end{algorithmic}
\end{algorithm}

Note that on line 11 of \Cref{alg:pseudo} we accept $\Delta_{k:k+d}$ as an estimate
of $\varepsilon_k$. However, at this moment we can compute a better approximation
to~$\varepsilon_k$ given by $\Delta_{k:k+d+1}$ which should be used in practical computations.
Nevertheless, for the sake of consistency with numerical experiments where we check the inequality \eqref{eqn:tolerance}, we store here $\Delta_{k:k+d}$. On line 11 we can also store
the corresponding value of $d_k=d$.
% that we plot in numerical experiments.
On line 15, the latest estimate $\Delta_{k-1:k+d}$ (after the update of~$k$ on line 12) can be used as a guaranteed lower bound and $\Delta_{k-1:k+d}/ (1 - \tau)$ as a heuristic upper bound on~$\varepsilon_{k-1}$.
Note also that on lines 7-15, it always holds $k +d = \ell-1$ so that
$\cfgq_{k+d+1} = \cfgq_{\ell}$. % and $\Delta_{k:k+d} = \Delta_{\ell-d-1:\ell-1}$.
If one wants to fix the smallest value of~$d$
to guarantee that information from at least, say~$d_{\min}$,
forthcoming iterations is always used to approximate~$\AN{k}$,
then one can replace the condition $d\geq 0$ on line 10 by the condition $d\geq d_{\min}$.

\section{Modifications of the algorithms for preconditioned CG}
 \label{sec:PCG}
In the standard view of preconditioning, the CG method is thought of as being applied to a ``preconditioned'' system
\begin{eqnarray}
 \label{EQP01}
        &&\hat{A} \hat{x} = \hat{b}, \qquad
        \hat{A} = L^{-1} A L^{-T},\quad \hat{b} =L^{-1} b,
\end{eqnarray}
where $L$ represents a  nonsingular (eventually lower triangular) matrix. Denoting the corresponding CG coefficients and vectors
with a hat and defining
\begin{eqnarray}
 \label{EQD01}%\qquad
       x_k \equiv L^{-T}\hat{x}_k,\ \
        r_k \equiv L\,\hat{r}_k,\ \
        p_k\equiv L^{-T}\hat{p}_k,\ \
        z_k \equiv L^{-T} L^{-1} r_k \equiv M^{-1}r_k,\nonumber
\end{eqnarray}
    (here $x_k$ and $r_k$ represent the approximate solution and residual for the original problem $Ax=b$),
    we obtain the standard version of the preconditioned CG (PCG) method which involves only $M=LL^T$; for more details see,
    e.g., \cite{Me1999, StTi2005, B:Me2006}.
The preconditioner $M$
should be chosen such that the linear system with the matrix $M$ is easy to solve, while the matrix $L^{-1} A L^{-T}$ ensures fast convergence of PCG.

\begin{algorithm}[ht]
\caption{Preconditioned CG (PCG) algorithm} \label{alg:pcg}
\begin{algorithmic}[0]
\STATE \textbf{input} $A$, $b$, $x_{0}$, $M$
\STATE $r_{0}=b-Ax_{0}$
\STATE $z_{0} = M^{-1}r_{0}$, $p_{0}=z_{0}$
%, $\hat{\cfa}_{0}=\frac{z_{0}^{T}r_{0}}{p_{0}^{T}Ap_{0}}$
\FOR{$k=0,\dots$ until convergence}
\STATE $\hat{\cfa}_{k}=\frac{z_{k}^{T}r_{k}}{p_{k}^{T}Ap_{k}}$
\STATE $x_{k+1}=x_{k}+\hat{\cfa}_{k}p_{k}$
\STATE $r_{k+1}=r_{k}-\hat{\cfa}_{k}Ap_{k}$
\STATE Solve $M z_{k+1}= r_{k+1}$
     \hspace*{15mm}%
        \smash{$\left.\begin{array}{@{}c@{}}\\{}\\{}\\{}\\{}\\{}\\{}\end{array}
        \right\} \begin{tabular}{l}{\tt pcgiter(k)}\end{tabular}$}
\STATE $\hat{\cfb}_{k+1}=\frac{z_{k+1}^{T}r_{k+1}}{z_{k}^{T}r_{k}}$
\STATE $p_{k+1}=z_{k+1}+\hat{\cfb}_{k}p_{k}$
 \ENDFOR
\end{algorithmic}%
\end{algorithm}

Since
$$
        \|\hat{r}_k\|^2 = r_k^T L^{-T}L^{-1}r_k = r_k^T M^{-1} r_k = z_k^T r_k\,
$$
and
$$
    \|\hat{x}-\hat{x}_k\|_{\hat{A}}^2     =
                (L^Tx-L^Tx_k)^TL^{-1}AL^{-T}(L^Tx-L^Tx_k)=\|x-x_k\|_{A}^2,
$$
the $A$-norm of the error vector $x-x_k$ in PCG can be estimated similarly as in classical CG.
In particular,
    \eqref{eqn:dlocaldecrease} takes the form
\begin{equation}
 \label{eqn:decreasePCG}
 \AN{k} \;=\; \hat{\cfgq}_{k:k+d} \;+\; \AN{k+d+1}\,,
\end{equation}
where
\begin{equation}
 \label{eqn:DeltaPCG}
 \hat{\cfgq}_{k:k+d} \equiv \sum^{k+d}_{j=k} \hat{\cfa}_{j} z_j^T r_j .
\end{equation}
Therefore, in PCG we can compute the lower bounds
using the PCG coefficients $\hat{\cfa}_{k}$ and inner products $z_k^T r_k$ (instead of using $\Vert \hat{r}_k\Vert^2$) that are computed anyway in the forthcoming PCG iterations. In other words, the lower bound \eqref{eqn:DeltaPCG} is still easy and cheap to evaluate. Note that the safety factor $S$ is now bounded by the condition number of the preconditioned matrix; see \cref{lem:kappa}.

%\section{Comments and possible improvements}\label{sec:extensions}
% In this section we comment on several modifications that
%can be helpful in more complicated cases.

\section{Improvements}
 \label{sec:extensions}
In this section we comment on several modifications that can be helpful in some difficult cases.
\subsection{The initial choice of~$d$}
 \label{sec:adaptd_init}
\cref{alg:pseudo} is very simple, and it provides fairly good results
in most of the situations; see \Cref{sec:experiments} for numerical experiments.
However, in some cases we can observe a long quasi-stagnation phase of the $A$-norm of the error vector during the initial iterations of (P)CG; see, e.g., the example {\tt s3dkt3m2} in \Cref{fig:numexp_03}. Then, it can happen that
the strategy suggested in \cref{alg:pseudo}
may not detect the quasi-stagnation, and the chosen value of~$d_k$ is typically significantly smaller than the ideal~$\widetilde{d}_k$. As a result, the $A$-norm of the error vector can be  underestimated by several orders of magnitude.
The large underestimation can last typically until (P)CG starts to converge faster.
%
%As we will see in the numerical experiments, the above strategy provides fairly good results. However, it has one possible drawback, the strategy may not identify a stagnation of the (P)CG in the initial iterations. If this happens, the chosen value of~$d(k)$ is typically significantly smaller than~$\widetilde{d}(k)$ and the error is underestimated. This can last until~$\widetilde{d}$ drops close to the current value of~$d$, typically until the (P)CG starts to converge faster.
%
To overcome this difficulty, we propose a safer strategy for the choice of the initial value of $d$ rather than starting from $d=0$.

Our aim is to find $d>0$ such that
%\begin{equation}\label{eq:initd}
\[
\frac{\varepsilon_{d}}{\varepsilon_{0}} < \tau,
\]
%\end{equation}
where $\tau$ is the prescribed tolerance. While $\varepsilon_{0}$ can be approximated as above by its lower
bound $\cfgq_{0:d}$, we need to obtain a convenient approximation to $\varepsilon_{d}$, ideally an upper bound.
%to satisfy (at least approximately) the inequality \eqref{eq:initd}.

If some underestimate of the smallest eigenvalue $\lambda_{\min}$ of the (preconditioned) system matrix is known, one can use upper bounds on the  error $\varepsilon_{d}$ discussed in~\cite{MeTi2019}.
In particular, if $0<\mu\leq \lambda_{\min}$ is given, then
\begin{equation}
 \label{eq:upper}
 \varepsilon_{d} < \cfa_{d}^{\smu}\|r_{d}\|^{2} < \frac{\|r_{d}\|^{2}}{\mu}\frac{\|r_{d}\|^{2}}{\|p_{d}\|^{2}}
\end{equation}
where $\cfa_{d}^{\smu}$ can be computed recursively using
\begin{equation}
 \label{eq:gamma}
 \cfa_{j+1}^{\smu}=\frac{\left(\cfa_{j}^{\smu}-\cfa_{j}\right)}{\mu\left(\cfa_{j}^{\smu}-\cfa_{j}\right)+\cfb_{j+1}},\quad\cfa_{0}^{\smu}=\frac{1}{\mu},
 \quad j=1,\dots,d-1.
\end{equation}
While the first bound in \eqref{eq:upper},
\begin{equation}
 \label{eq:GR}
 \omega_d^{\klein{(\mu)}} \equiv \cfa_{d}^{\smu}\|r_{d}\|^{2}
\end{equation}
cannot be used as an approximation to $\varepsilon_{d}$ if $\mu>\lambda_{\min}$, the second bound
in \eqref{eq:upper} can still serve as an approximation to
$\varepsilon_{d}$ for any $\mu\approx \lambda_{\min}$; see \cite{MeTi2019}.
Note that there are several applications in the context
of numerical solving of PDE's,
see, e.g., \cite{GeMaNiBjSt2019,GeNiBjSt2020,KuPu2020}, where a tight underestimate of $\lambda_{\min}$ can always be determined a priori.

If no a priori information about the smallest eigenvalue
$\lambda_{\min}$ is known, we can approximate it from the (P)CG process as described in \cite{MeTi2019}.
%
%
%$$
%{\frac{\varepsilon_{d}}{\varepsilon_{0}}\approx}
%\frac{\MT{d}}{\cfgq_{0,d}} < \tau,
%$$
%
%in this section an additional strategy (to be combined with the strategy of Section~\ref{sec:adaptd_later}), which adaptively chooses the initial value of~$d$. The strategy uses an (approximate) upper bound on the $A$-norm of the error from~\cite{MeTi2019}, which is presented below.
%
In particular, denoting the initial values
\[
{\rho}_{0}=\cfa_0, \ {\tau}_{0}={\rho}_{0},
\ {\sigma}_{0}=0, {s}_{0}=0,\ {c}_{0}=1,\ \pi_0 = 1,
\]
by {\tt incremental()}, the smallest eigenvalue $\lambda_{\min}$ of the (preconditioned) matrix can be approximated
incrementally from the (P)CG coefficients, using the recurrences
denoted as {\tt incremental($k$)}:
\begin{eqnarray*}
{\sigma}_{k}&=&-\sqrt{\cfa_{k}\frac{\cfb_{k}}{\cfa_{k-1}}}\left({s}_{k-1}{\sigma}_{k-1}+{c}_{k-1}{\tau}_{k-1}\right),\ \
{\tau}_{k}=\cfa_{k}\left(\cfb_{k}\frac{{\tau}_{k-1}}{\cfa_{k-1}}+1\right),\\
{\chi}_{k}^{2}&=&\left({\rho}_{k-1}-{\tau}_{k}\right)^{2}+4{\sigma}_{k}^{2},\ \
{c}_{k}^{2}=\frac{1}{2}\left(1-\frac{{\rho}_{k-1}-{\tau}_{k}}{{\chi}_{k}}\right),\\
{\rho}_{k}&=&{\rho}_{k-1}+{\chi}_{k}{c}_{k}^{2},\\
{s}_{k}&=&\sqrt{1-{c}_{k}^{2}},\ \
{c}_{k}=|{c}_{k}|\,\mathrm{sign}({\sigma}_{k}),\\
\mu_k &=& {\rho}_{k}^{-1},\\
\pi_k &=& \frac{\pi_{k-1}}{\pi_{k-1}+\cfb_k}.
\end{eqnarray*}
%\cred{{\bf This is not very good during the first iterations}}
%
Using the idea of incremental norm estimation, the above algorithm aims to approximate the smallest Ritz value.
The updated value $\mu_k$ estimates the smallest Ritz value from above, and, therefore, $\mu_k > \lambda_{\min}$.
Numerical experiments in \cite{MeTi2019} predict that for $k$ sufficiently large
$\mu_k$ usually approximates the smallest Ritz value to one or two valid digits.
%estimates the smallest Ritz value from above, and it can serve as an %approximation to $\lambda_{\min}$.
Since $\pi_k = \Vert r_k \Vert^2/ \Vert {p_k} \Vert^2$, we can define the quantity
\begin{equation}
 \label{eqn:aUB}
 \MT{k} \equiv \frac{\pi_k}{\mu_k}\|r_k\|^2
\end{equation}
with the idea of approximating the rightmost upper bound in \eqref{eq:upper};
%that approximates $\varepsilon_{k}$
see \cite{MeTi2019}.
In the initial stage of convergence, the smallest Ritz value (and, therefore, also $\mu_k$) is a poor approximation to the smallest eigenvalue $\lambda_{\min}$. Therefore, we cannot expect that $\mu_k \approx\lambda_{\min}$, and
$\MT{k}$ typically underestimates $\varepsilon_{k}$ by several orders of magnitude.
%Since $\mu_k > \lambda_{\min}$, we cannot guarantee
%that $\MT{k}$ represents an upper bound of $\AN{k}$.
However, as soon as the smallest Ritz value
starts to be a fair approximation of $\lambda_{\min}$, one can expect that $\MT{k}$ is an upper bound on $\AN{k}$. Note that if $\mu \leq \lambda_{\min}$ is available, then one can replace~$\mu_k$ in \eqref{eqn:aUB} by $\mu$, and
$\MT{d}$ would represent a guaranteed upper bound on $\AN{d}$.

\begin{algorithm}[ht]
\caption{CG with the adaptive choice of $d$ and the initial phase} \label{alg:pseudofull}
\algsetup{indent=2em}
\begin{algorithmic}[1]

\STATE \textbf{input} $A$, $b$, $x_{0}$, $\tau$, $\mathrm{TOL}$

\STATE $r_0 = b-Ax_0$, $p_0=r_0$
\STATE ${d=0}$, $k=0$
\STATE {\tt incremental()}

\STATE {\tt initial} = true

\FOR{$\ell=0,\dots,$}

\STATE {\tt cgiter($\ell$)}

\IF{{\tt initial}}
%\STATE \hspace{2cm}\textcolor{gray}{\emph{strategy of Section~\ref{sec:adaptd_init}:}}

\STATE {\tt incremental($\ell$)}
\STATE compute~$\MT{d}$ using~\eqref{eqn:aUB}
\IF{\eqref{eqn:initial}}
	\STATE {\tt initial} = false
\ELSE
	\STATE ${d=d+1}$
\ENDIF

%\IF{\NOT {\tt init\_{}phase}}
\ELSE
%\STATE \textcolor{gray}{\emph{strategy of Section~\ref{sec:adaptd_later}:}}
\STATE use code on lines 8-16 of \Cref{alg:pseudo}
\ENDIF
\ENDFOR

\end{algorithmic}
\end{algorithm}

As an improved strategy, we suggest  to choose the initial value of $d$ such that
\begin{equation}
 \label{eqn:initial}
%\textcolor{gray}{\frac{\varepsilon_{d}}{\varepsilon_{0}}\approx}
 \frac{\varepsilon_{d}}{\varepsilon_{0}} \approx \frac{\MT{d}}{\cfgq_{0:d}} < \tau;
\end{equation}
see \Cref{alg:pseudofull}.
As already mentioned, in the initial stage of convergence, the quantity
$\MT{d}$~can underestimate $\varepsilon_{d}$ by several orders of magnitude.
Despite this fact, the ratio ${\MT{d}}/{\cfgq_{0:d}}$ often represents
an upper bound on $\varepsilon_{d}/\varepsilon_{0}$. To get an idea why it is so, let us assume that the
error stagnates in the initial stage of convergence.
Then both $\MT{d}$ as well as $\cfgq_{0:d}$ substantially underestimate the target quantities, but the underestimation using $\MT{d}$ is relatively smaller than the one based on $\cfgq_{0:d}$.
As a result, the ratio of the two underestimates is an upper bound on
$\AN{d}/\AN{0}$.
The above considerations are purely heuristic, but work satisfactory in all of our experiments.

%
%
%While $\cfgq_{0,d}$ is a guaranteed lower bound on $\varepsilon_{0}$,

\subsection{Using (the approximation of) the upper bound} \label{sec:GRupper}
The quantity $\MT{k}$ defined by \eqref{eqn:aUB} can also be used to approximate the ratio in
\eqref{eqn:tolerance2}, leading to another strategy for the adaptive choice of $d$. However, as already mentioned in
\Cref{sec:adaptd_later}, even if we know the smallest eigenvalue to a high accuracy, the bounds and approximations presented in
\eqref{eq:upper} and \eqref{eqn:aUB} often do not provide relevant information on the current error $\AN{k}$ in the final stage of convergence.
In some cases they can be delayed several hundreds of iterations over $\AN{k}$. As a result, an adaptive strategy based on approximating
$\AN{k+d+1}$ by $\MT{k+d+1}$ in  \eqref{eqn:tolerance2} would use values of $d$ larger than necessary.

\begin{figure}[!htbp]
 \label{fig:02}
\centering
 \includegraphics[width=9cm]{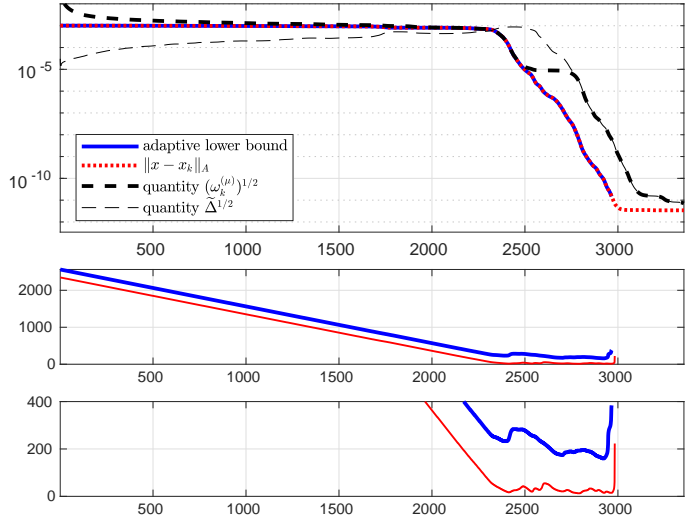}
\caption{Matrix {\tt s3dkt3m2}, top part:
the $A$-norm of the error vector $x-x_k$ (dots), the square root of the upper bound
$\omega_d^{\klein{(\mu)}}$ (bold dashed), the quantity $\MT{k}^{1/2}$ (dashed),
and the lower bound with the adaptive choice of $d$ based on
replacing $\AN{k+d+1}$ in \eqref{eqn:tolerance2} by $\MT{k+d+1}$ (blue solid curve). Bottom parts: the adaptive value $d_k$ (blue bold solid curve) and the ideal value~$\widetilde{d}_k$ (red solid curve).}
\end{figure}
This is clearly demonstrated in  \Cref{fig:02}, where the example {\tt s3dkt3m2} is considered. In the top part of the figure we
 plot the $A$-norm of $x-x_k$ (dotted curve), the quantity $\MT{k}^{1/2}$ (dashed curve) defined using \eqref{eqn:aUB},
and the square root of the Gauss--Radau upper bound $\omega_k^{\klein{(\mu)}}$ defined in \eqref{eq:GR} (bold dashed curve).
%
%\begin{equation}\label{eqn:GR}
%    (\cfa_{k}^{\smu})^{1/2}\|r_{k}\|.
%\end{equation}
%
Note that to compute the upper bound $\omega_k^{\klein{(\mu)}}$, we have to provide a tight
underestimate $\mu$ of the smallest eigenvalue of the preconditioned matrix.
For experimental reasons, we computed the smallest Ritz value at an iteration~$k$ in which the ultimate level of accuracy was reached, and got a very accurate approximation to the smallest eigenvalue of the preconditioned matrix.
Then we defined $\mu$ to be this computed approximation to the smallest eigenvalue divided by $1+10^{-4}$. We set $\tau=0.25$.
% \cred{{\bf What was it? and what was the value of $\tau$?}}
If we use the quantity $\MT{k+d+1}$ to approximate $\AN{k+d+1}$
in the ratio \eqref{eqn:tolerance2}, and choose $d$ such that the approximated radio is less than $\tau$, we obtain very tight lower bounds (blue solid curve). However, in the final stage of convergence that is very important for stopping the iterations, such an  adaptively chosen $d$ differs from the ideal $\widetilde{d}_k$ by about 200 (the bottom part). In other words, we compute about 200 iterations of P(CG) more than necessary to reach a given
level of accuracy.
Note that very similar results will be obtained when using the Gauss--Radau upper bound \eqref{eq:GR} to approximate $\AN{k+d+1}$ in the ratio \eqref{eqn:tolerance2}, since {normal and bold dashed curves almost coincide in the final stage of convergence. We will show how to fix these problems with the Gauss--Radau upper bound in a forthcoming paper.

In summary, the strategy for the adaptive choice of $d$ based
on replacing $\AN{k+d+1}$ by $\MT{k+d+1}$ or by $\omega_{k+d+1}^{\klein{(\mu)}}$ in \eqref{eqn:tolerance2}, seems to be a safe strategy
in the sense that it usually produces $d$ larger than their ideal counterparts. For such $d$'s, the inequality \eqref{eqn:tolerance2} is satisfied.
%
%that produces very tight underestimates of $\AN{k}$.
%
On the other hand, it can use values of $d$ that are significantly larger than
necessary to reach
the prescribed accuracy of the estimate $\Delta_{k:k+d}$.
In particular, we observed unnecessarily large values of $d$ in the later (P)CG convergence phase.
Therefore, it is better to use the strategy of \Cref{alg:pseudo}, eventually combined with the strategy for choosing the initial value of $d$.

\section{Numerical experiments}
 \label{sec:experiments}

In the following numerical experiments we consider a set of test problems from the SuiteSparse Matrix collection, listed in \Cref{tab:test_problems}. We believe that the selected problems represent the most relevant scenarios one can face in practical computations. The experiments were run using double precision in MATLAB R2019b.

\begin{table}[htp]
\centering
\begin{tabular}{>{\tt}lrcc}
\textrm{name} & size & rhs $b$ & precond.~$M=LL^T$ \\ \hline
bcsstk02 & 66 & \rdelim \} {3}{10em}[{~equal components}] & --- \\

bcsstk04 & 132 &  & ---\\
%bcsstk09 & 1083 &  & ---\\
bcsstk09 & 1083 &  & ict(1e-3, 1e-2)\\

%bcsstk26 & 1922 &  & ict(1e-4, 1e-2)\\
%bcsstk25 & 15\,439 &  & ict(1e-4, 1e-2)\\
s3dkt3m2 & 90\,449 & comes with the matrix & ict(1e-5, 1e-2)\\
s3dkq4m2 & 90\,449 &  \rdelim \} {5}{10em}[~$\mbox{rand}(-1,1)$] & ict(1e-5, 1e-2)\\ %\multirow{5}{*}{$b_j = \mbox{rand}(-1,1)$}
pwtk & 217\,918 &  & ict(1e-5, 1e-1)\\
af\_shell3 & 504\,855 &  & zero-fill\\ %nofill(1e-5,1e-1)
tmt\_sym & 726\,713 &  & zero-fill\\
ldoor & 952\,203 &  & zero-fill\\[0.4cm]
\end{tabular}
\caption{A set of test problems from the SuiteSparse Matrix collection.}
\label{tab:test_problems}
\end{table}

When it is not provided with the matrix, the right-hand side $b$
is chosen such that $b$ has equal components in the eigenvector basis
 or such that the entries of $b$ are randomly generated in the interval $(-1,1)$. In both cases, $b$ is normalized to have the unit norm $\Vert b \Vert =1$.
The right-hand side for {\tt s3dkt3m2} with only
the last element nonzero comes with the application; see \cite{Ko1999}.

The preconditioners are determined by the incomplete Cholesky
factorization (using the {\tt ichol} command),
either with zero-fill or with threshold dropping (ict)
where the first parameter is the
drop tolerance and the second parameter is the global diagonal shift;
for more details see the MATLAB documentation.

In \Crefrange{fig:numexp_01}{fig:numexp_04} we test the accuracy of the lower bound $\Delta_{k:k+d_k}^{\klein{1/2}}$
as an approximation of the $A$-norm of the error vector $x-x_k$, $\| x - x_k \|_A = \AN{k}^{\klein{1/2}}$, 
where $d_k$ is determined as in \Cref{alg:pseudo}. In the top parts of \Crefrange{fig:numexp_01}{fig:numexp_04} we plot
the lower bound $\Delta_{k:k+d_k}^{\klein{1/2}}$ (blue solid curve)
together with $\| x - x_k \|_A$ (red dots).
In the middle parts of figures we plot
the relative error
$$
\frac{\AN{k}-\cfgq_{k:k+d_k}}{\AN{k}}
$$
(blue solid curve) together with the prescribed tolerance $\tau = 0.25$ (dotted line); see~\eqref{eqn:tolerance}. Finally, in the bottom parts of \Crefrange{fig:numexp_01}{fig:numexp_04} we plot the value of~$d_k$
determined by \Cref{alg:pseudo},
and compare it with the ideal value $\widetilde{d}_k$ of \Cref{def:ideal}. Note that the ideal value $\widetilde{d}_k$ was determined using
the quantities $\AN{k}$ that are not known in practical computations.

\begin{figure}
\centering
\begin{minipage}[c]{0.49\linewidth}
\centering
{\tt bcsstk02}\\[0.2cm]
\includegraphics[width=\linewidth]{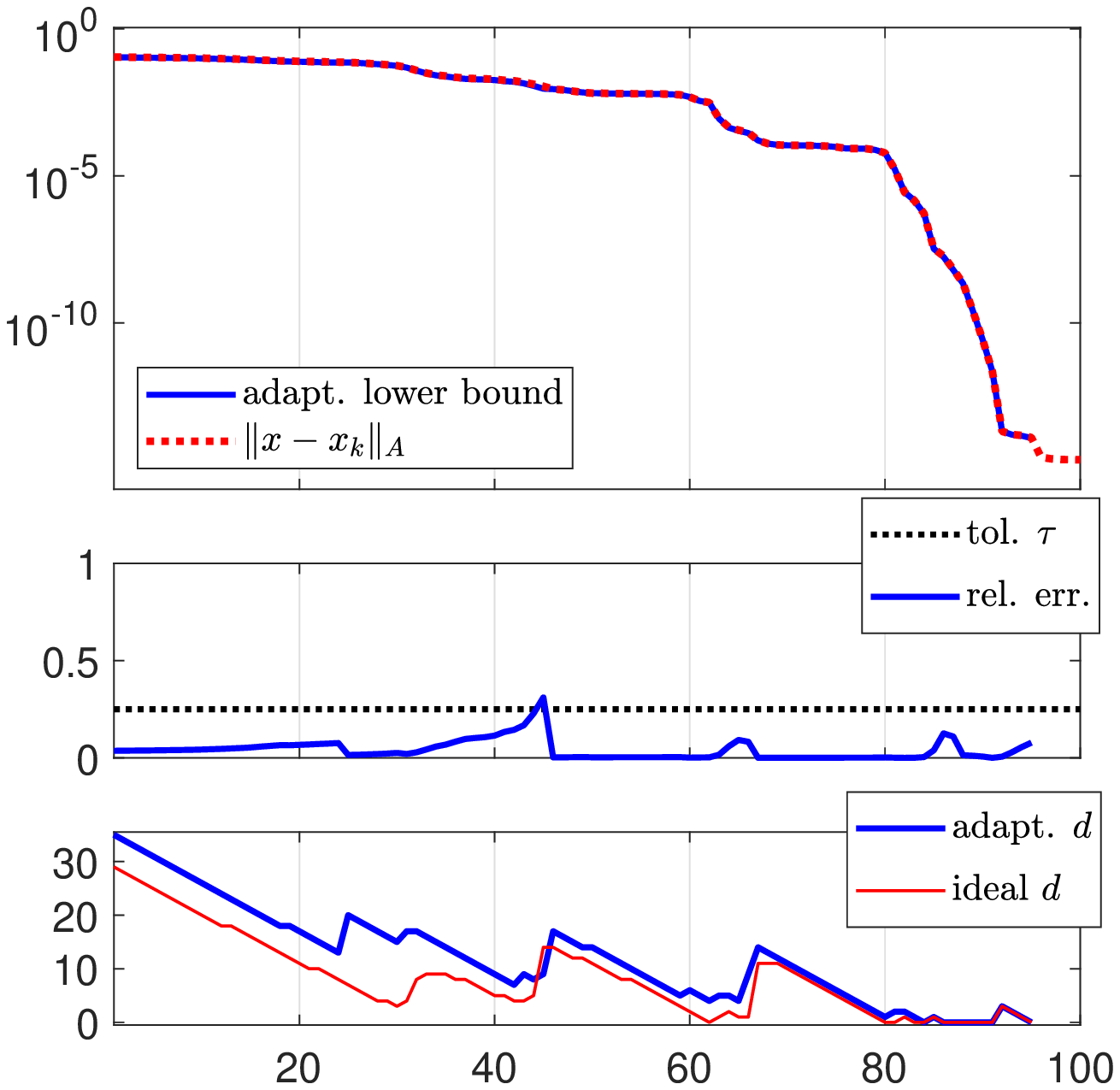}\\
CG iterations
\end{minipage}
 \hfill
\begin{minipage}[c]{0.49\linewidth}
\centering
{\tt bcsstk04}\\[0.2cm]
\includegraphics[width=\linewidth]{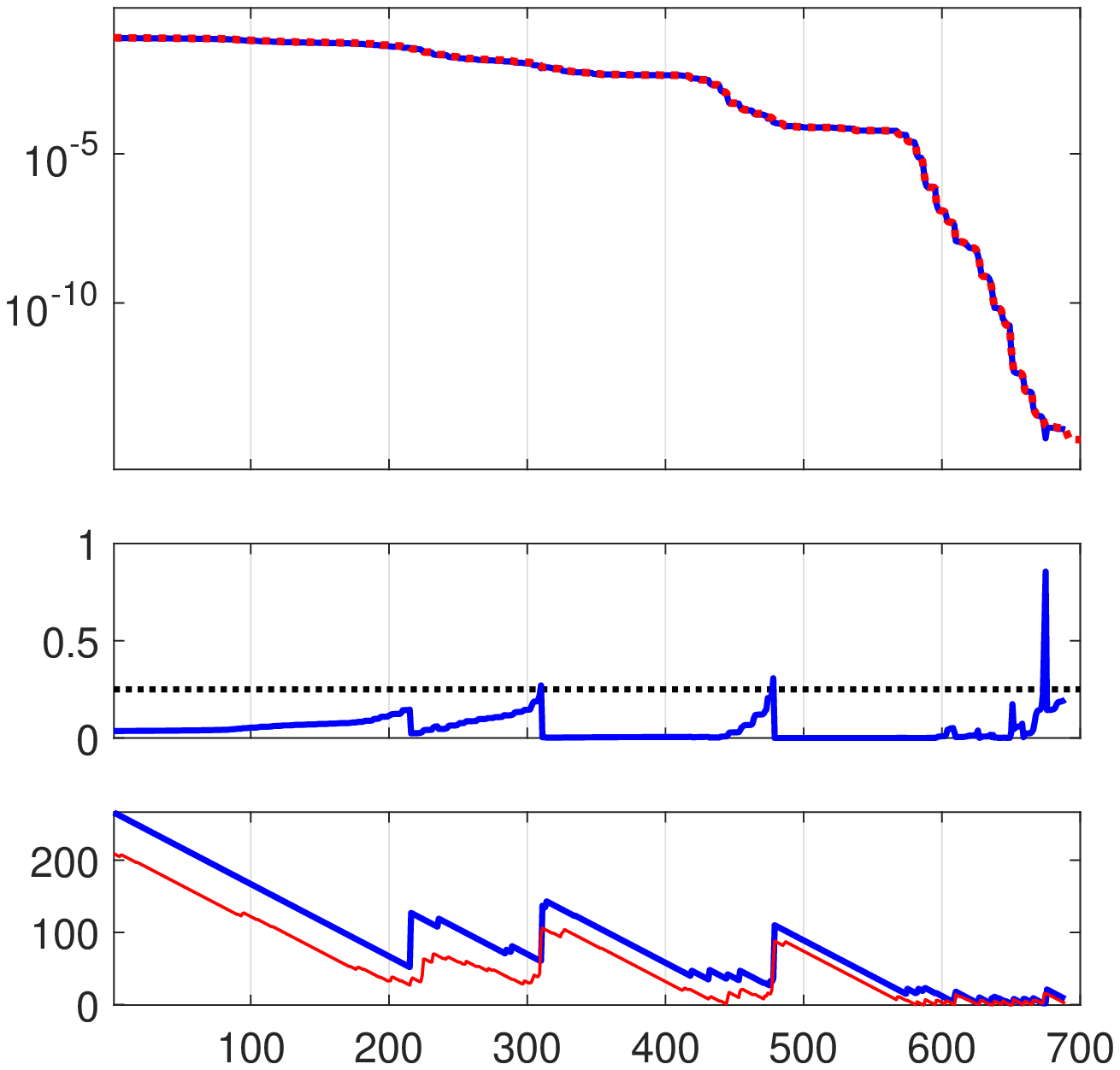}\\
CG iterations
\end{minipage}

\caption{Matrices {\tt bcsstk02} and {\tt bcsstk04}: the $A$-norm of the error vector $x-x_k$ and the adaptive lower bound (top part), the relative error and the prescribed tolerance from~\eqref{eqn:tolerance} (middle part), the value $d_k$ and the ideal value~$\widetilde{d}_k$ (bottom part).} \label{fig:numexp_01}
\end{figure}

For the first two test problems (\Cref{fig:numexp_01}), the unpreconditioned CG method needs
significantly more iterations than what is the size of the problem
to reach the ultimate level of accuracy.
In other words, convergence is substantially delayed due to finite precision arithmetic.
We choose these examples to demonstrate that
the estimates and techniques for the adaptive choice of $d$ work
well also in finite precision. This is actually no surprise.
Using results of \cite{StTi2002} we know that the estimation process is based on identities that do hold (up to some small inaccuracy) also during finite precision computations, despite the loss of orthogonality, until the ultimate level of accuracy is reached. One can observe that the required accuracy of estimates is reached and that the value of $d_k$ is  close to the ideal value $\widetilde{d}_k$ in almost all iterations.

\begin{figure}[htp]
\centering
\begin{minipage}[c]{0.49\linewidth}
\centering
{\tt bcsstk09}+no precond.\\[0.2cm]
\includegraphics[width=\linewidth]{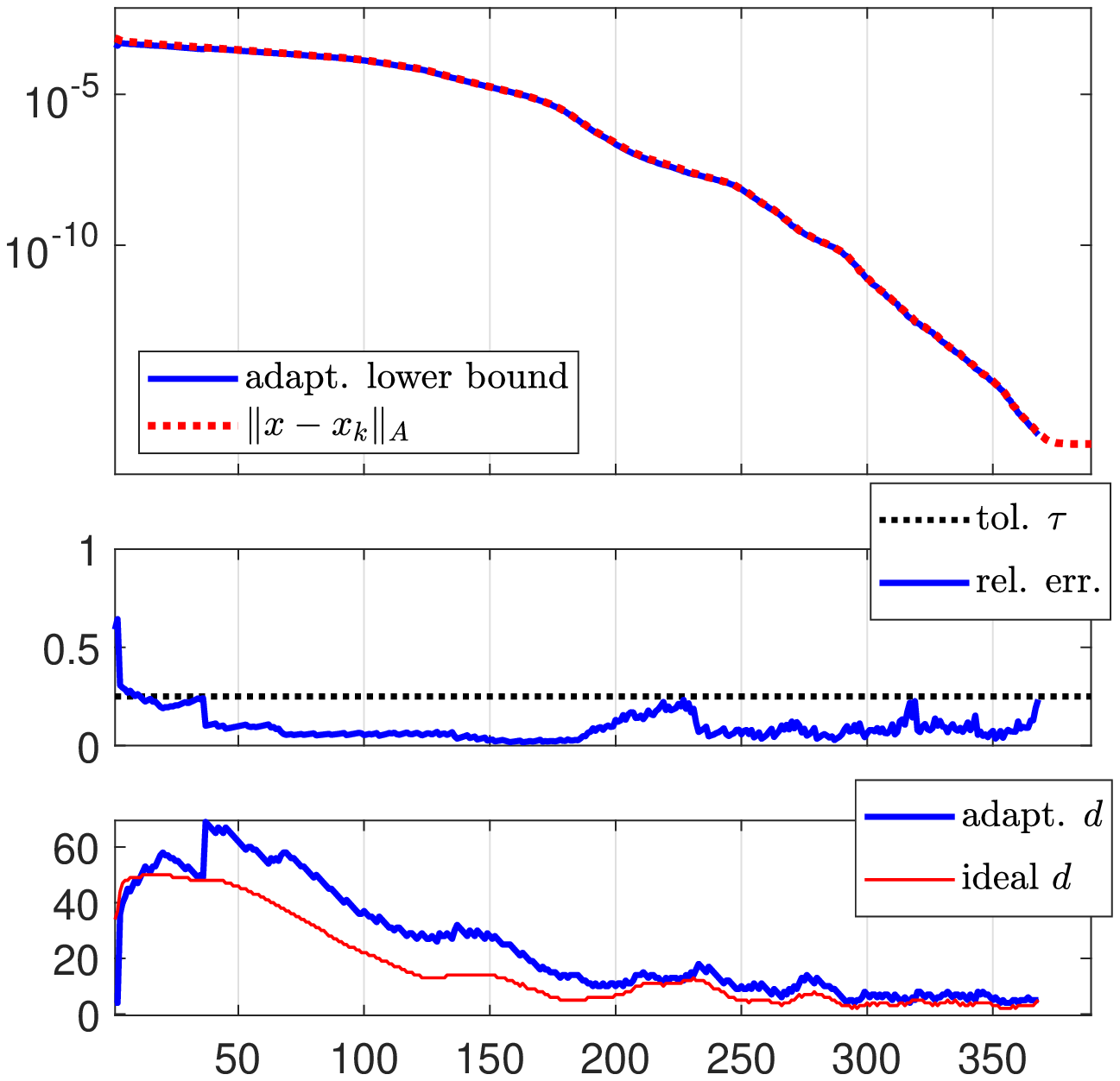}\\
CG iterations
\end{minipage}
 \hfill
\begin{minipage}[c]{0.49\linewidth}
\centering
{\tt bcsstk09}+ict(1e-3, 1e-2)\\[0.2cm]
\includegraphics[width=\linewidth]{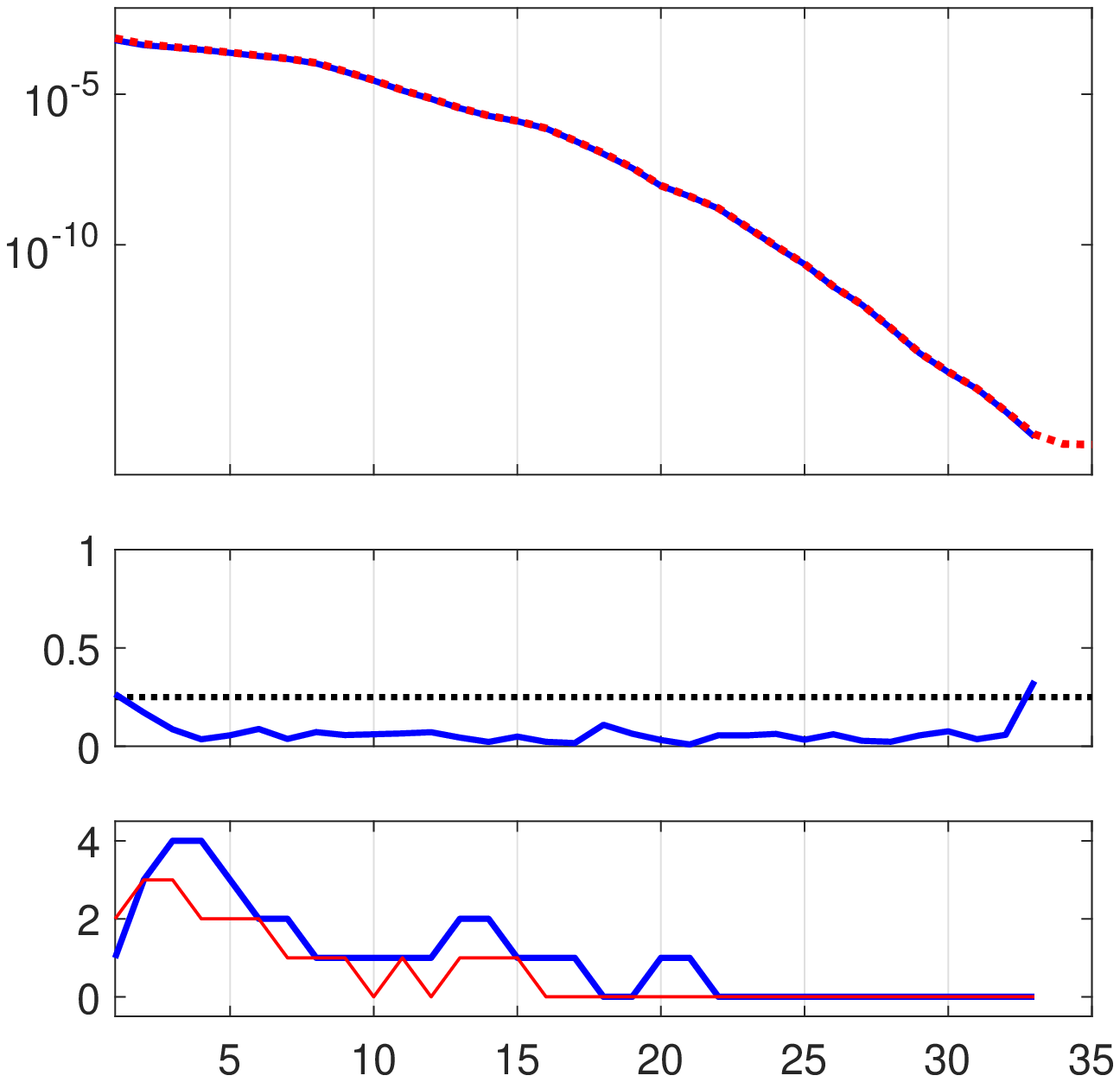}\\
PCG iterations
\end{minipage}

\caption{Matrix {\tt bcsstk09} without and with the preconditioner: the $A$-norm of the error vector $x-x_k$ and the adaptive lower bound (top part), the relative error and the prescribed tolerance from~\eqref{eqn:tolerance} (middle part), the value $d_k$ and the ideal value~$\widetilde{d}_k$ (bottom part).} \label{fig:numexp_02}
\end{figure}

In \Cref{fig:numexp_02} we test the performance of the adaptive procedure by considering the test problem {\tt bcsstk09} of size~$1083$ with and without a preconditioner. Similarly as in the previous example, the adaptive strategy of \Cref{alg:pseudo} works satisfactorily and provides estimates with the prescribed accuracy.

In the unpreconditioned case, the convergence is at first quite slow so that one needs larger values of $d_k$ (around 50) to reach the prescribed accuracy. As soon as convergence accelerates (around iteration 200), just a moderate value of $d_k$ (around 10) is needed.
Our strategy perfectly captures this behavior, though in the beginning
$d_k$ overestimates the ideal value $\widetilde{d}_k$ moderately. However, this should not represent
a serious drawback in practical computations since the iterations will probably be stopped after iteration 200, when $d_k$ and the ideal $\widetilde{d}_k$ almost coincide.
If stopped earlier, we just obtain a more accurate estimate than necessary.

In the preconditioned case, convergence is fast and the adaptive
strategy provides values of $d_k$ that are almost identical to
 $\widetilde{d}_k$. This is in particular true in the final convergence phase when $\widetilde{d}_k=0$ so that PCG uses only terms from the current iteration. Hence, in this example, the adaptive strategy works very well in the case of fast convergence and no unnecessary iterations are needed.

\begin{figure}
\centering
\begin{minipage}[c]{0.49\linewidth}
\centering
{\tt pwtk}\\[0.2cm]
\includegraphics[width=\linewidth]{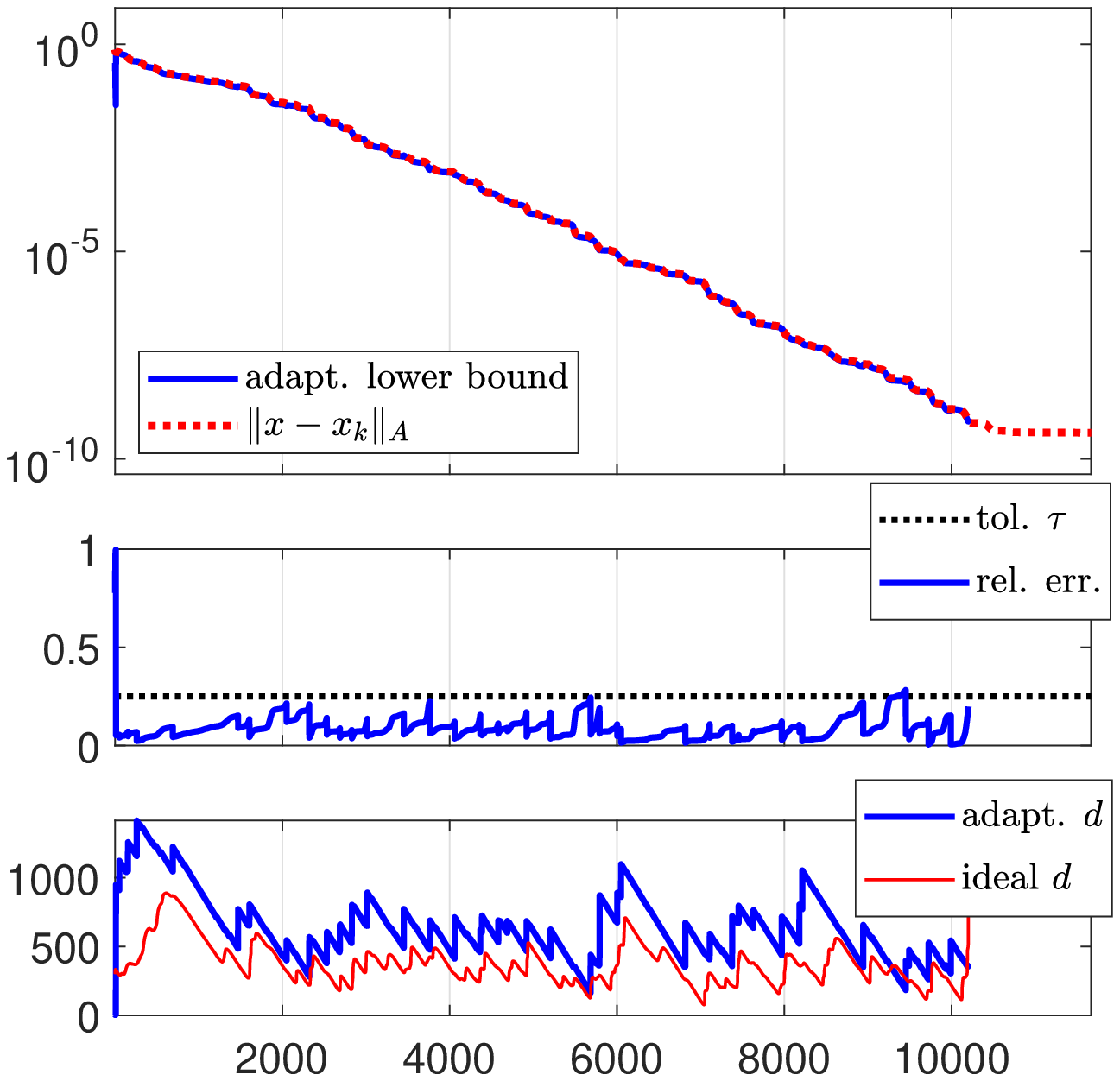}\\
PCG iterations
\end{minipage}
 \hfill
\begin{minipage}[c]{0.49\linewidth}
\centering
{\tt af\_shell3}\\[0.2cm]
\includegraphics[width=\linewidth]{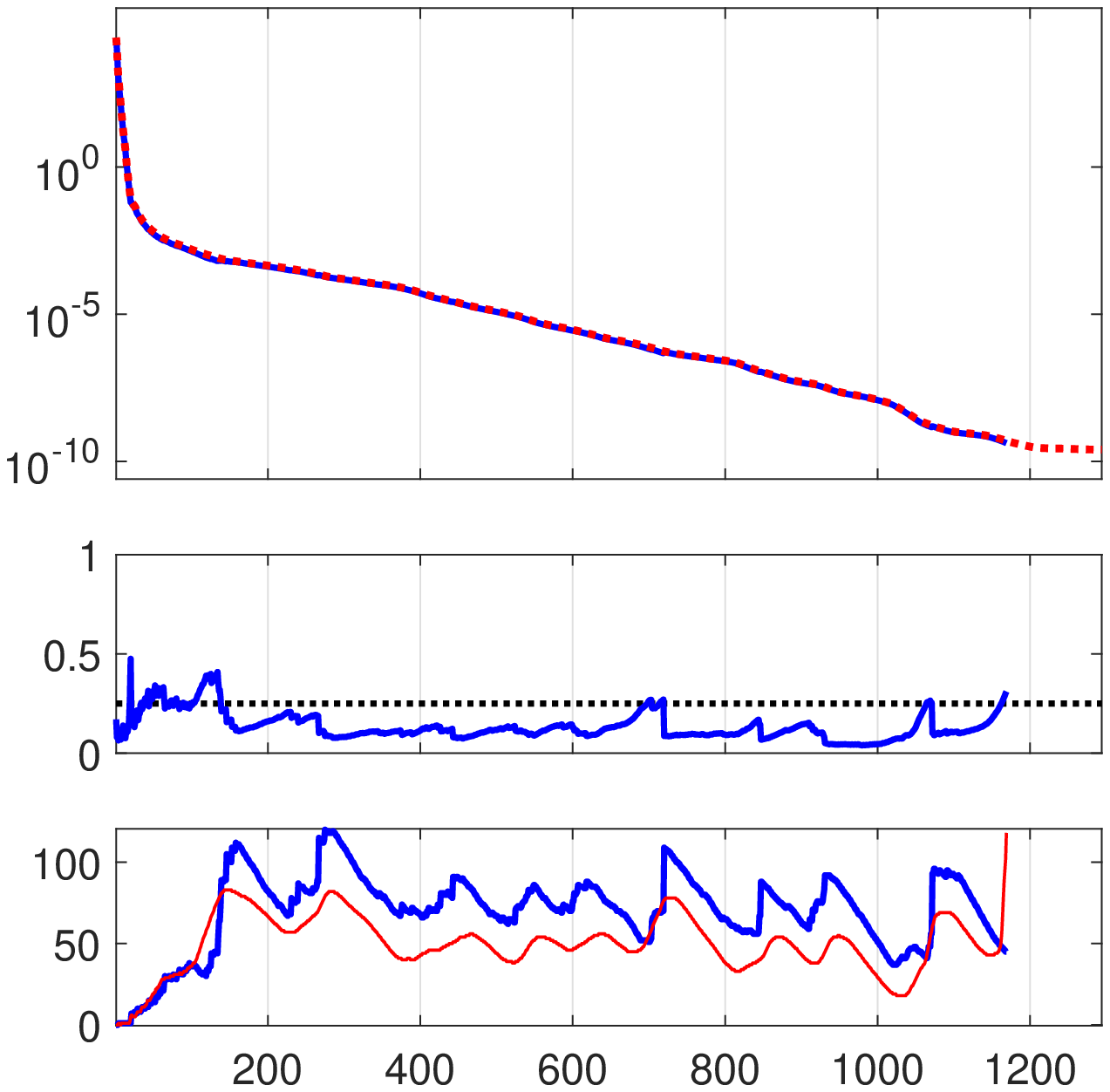}\\
PCG iterations
\end{minipage}
 \\[0.7cm]
\begin{minipage}[c]{0.49\linewidth}
\centering
{\tt tmt\_sym}\\[0.2cm]
\includegraphics[width=\linewidth]{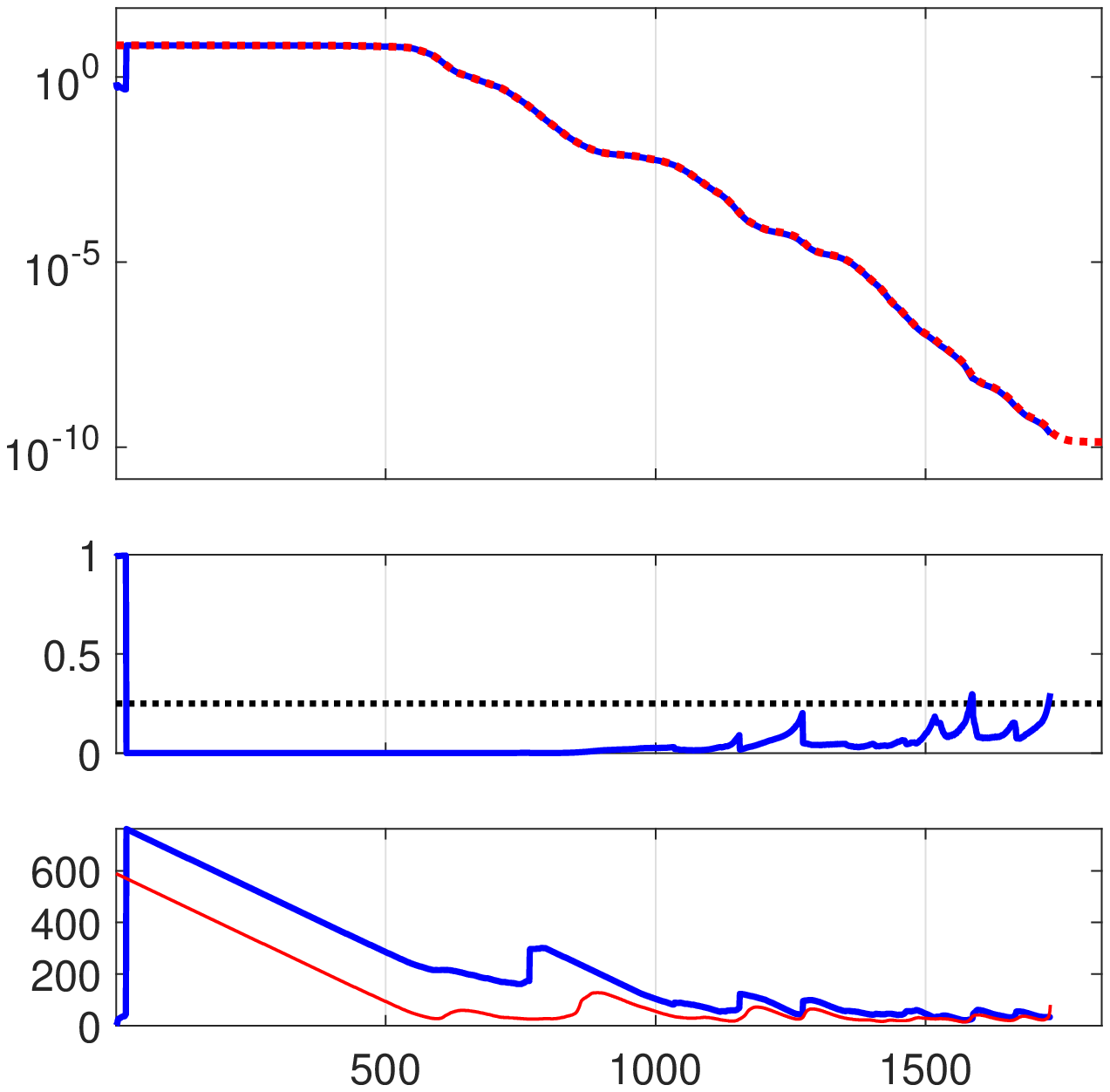}\\
PCG iterations
\end{minipage}
 \hfill
\begin{minipage}[c]{0.49\linewidth}
\centering
{\tt ldoor}\\[0.2cm]
\includegraphics[width=\linewidth]{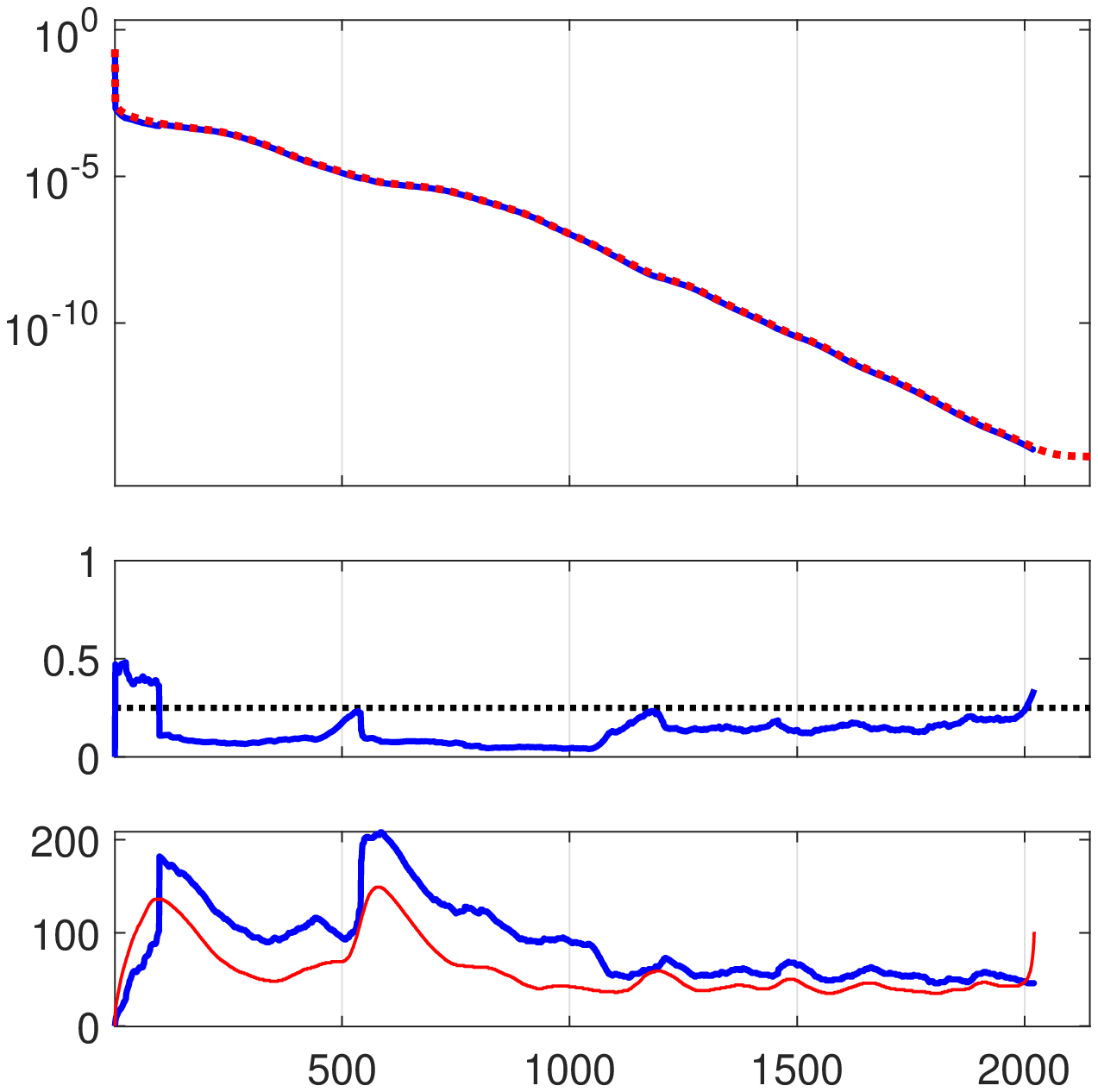}\\
PCG iterations
\end{minipage}

\caption{Matrices {\tt pwtk}, {\tt af\_shell3}, {\tt tmt\_sym}, {\tt ldoor}: the $A$-norm of the error vector $x-x_k$ and the adaptive lower bound (top part), the relative error and the prescribed tolerance from~\eqref{eqn:tolerance} (middle part), the value $d_k$ and the ideal value~$\widetilde{d}_k$ (bottom part).} \label{fig:numexp_03}
\end{figure}

In \Cref{fig:numexp_03} we consider quite large problems {\tt pwtk}, {\tt af\_shell3}, {\tt tmt\_sym}, and {\tt ldoor} that we are still able to handle using MATLAB on a personal computer. In all four cases, PCG was used to solve the systems.
We can again conclude that our adaptive strategy works satisfactorily
and provides tight bounds.
\Cref{alg:pseudo} can handle slow or fast convergence (as for {\tt pwtk} and {\tt ldoor}), the staircase convergence ({\tt pwtk} and {\tt tmt\_sym}) and (mostly) also initial quasi-stagnation phases ({\tt tmt\_sym}).
In some iterations, the value $d_k$ overestimates moderately the ideal value $\widetilde{d}_k$, providing more accurate adaptive estimates than required.
On the other hand, it is better to slightly overestimate the ideal value of
$\widetilde{d}_k$ than to underestimate it. As in the previous examples we observe that in iterations that are suitable for stopping the algorithm (a few orders of magnitude above the ultimate level of accuracy), the adaptively  determined  $d_k$ is very close to the ideal $\widetilde{d}_k$ so that no unnecessary iterations are needed to get the estimates with the required accuracy.

Note that in the test problem {\tt tmt\_sym}, the $A$-norm of the error vector exhibits quite a long initial quasi-stagnation phase (up to the iterations 500). In such cases, $d_k$ computed by \Cref{alg:pseudo} need not be close to
$\widetilde{d}_k$ in a few initial iterations. As a result, the lower bound $\Delta_{k:k+d_k}^{\klein{1/2}}$ can underestimate visibly the quantity of interest~$\varepsilon_k^{\klein{1/2}}$; see also the test problem
{\tt bcsstk04\_sym} (\Cref{fig:numexp_01}). In such cases,
the strategy for the initial choice of
described in  \Cref{sec:adaptd_init} should be used.
This strategy will be discussed in more detail in a later experiment.

\begin{figure}
\centering
%\begin{minipage}[c]{0.49\linewidth}
%\centering
%{\tt bcsstk26}\\[0.2cm]
%\includegraphics[width=\linewidth]{24-bcsstk26}\\
%PCG iterations
%\end{minipage}
% \hfill
%\begin{minipage}[c]{0.49\linewidth}
%\centering
%{\tt bcsstk25}\\[0.2cm]
%\includegraphics[width=\linewidth]{22-bcsstk25}\\
%PCG iterations
%\end{minipage}
% \\[0.7cm]
\begin{minipage}[c]{0.49\linewidth}
\centering
{\tt s3dkt3m2}\\[0.2cm]
\includegraphics[width=\linewidth]{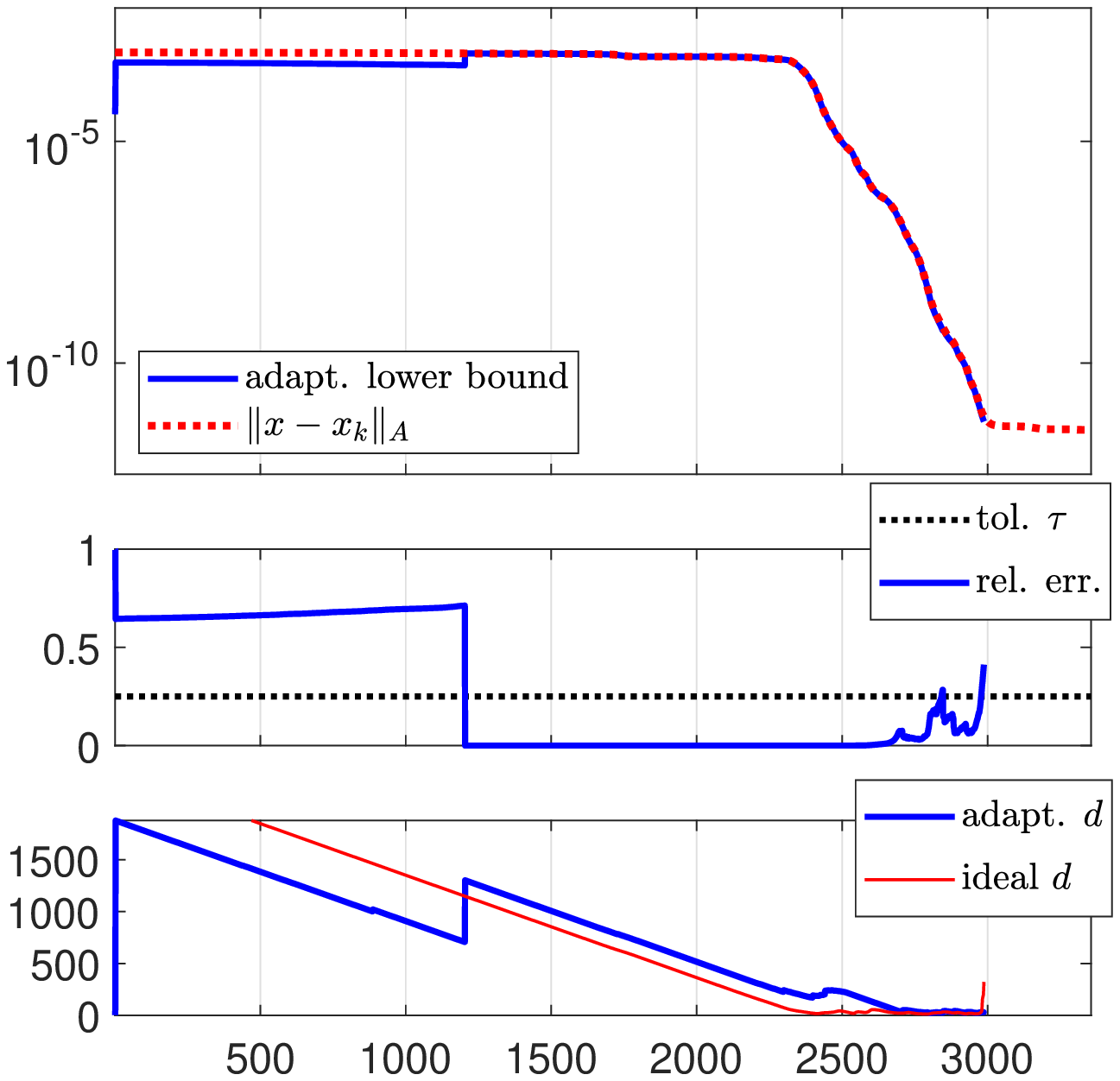}\\
PCG iterations
\end{minipage}
 \hfill
\begin{minipage}[c]{0.49\linewidth}
\centering
{\tt s3dkq4m2}\\[0.2cm]
\includegraphics[width=\linewidth]{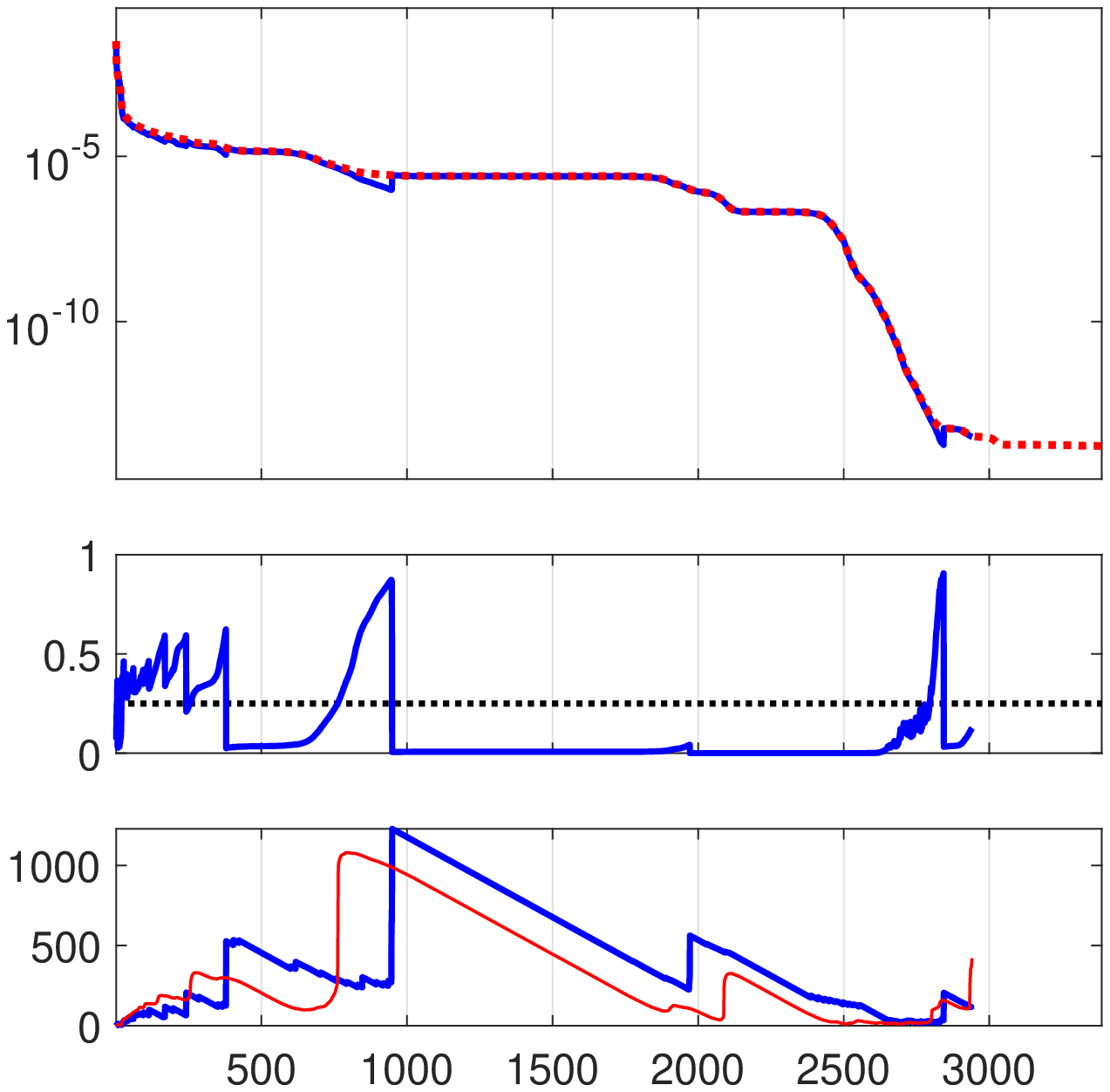}\\
PCG iterations
\end{minipage}

\caption{Matrices {\tt s3dkt3m2}, {\tt s3dkq4m2}: the $A$-norm of the error vector $x-x_k$ and the adaptive lower bound (top part), the relative error and the prescribed tolerance from~\eqref{eqn:tolerance} (middle part), the  value $d_k$ and the ideal value~$\widetilde{d}_k$ (bottom part).} \label{fig:numexp_04}
\end{figure}

In \Cref{fig:numexp_04} we show the most difficult cases we have found when testing \Cref{alg:pseudo} for the adaptive choice of $d_k$.
In the test problems {\tt s3dkt3m2} and {\tt s3dkq4m2},
the $A$-norm of the error vector exhibits various phases of the convergence (sublinear, superlinear, quasi-stagnation) that alternate.
The quasi-stagnation phase can happen at the very beginning as for {\tt s3dkt3m2}
but also in later iterations as for {\tt s3dkq4m2}. Such a complicated
convergence behaviour need not be fully captured by \Cref{alg:pseudo}.
At some iterations, the resulting $d_k$ underestimates significantly the ideal value~$\widetilde{d}_k$ and, therefore, we observe visually a considerable underestimation of the quantity of interest. While for
{\tt s3dkt3m2} we are able to fix the problem using
a suitable initial $d_0$, see \Cref{sec:adaptd_init}
and \Cref{fig:numexp_06_UB}, for the problem {\tt s3dkq4m2}
we were not able to find an improvement that would
prevent an underestimation close to the iteration 1000.

Nevertheless, though the strategy of \Cref{alg:pseudo} applied
to the test problems shown in \Cref{fig:numexp_04} does not always
provide estimates of $\|x-x_k\|_A$ with a required accuracy, it can still be considered as satisfactory. It provides lower bounds that are close to the quantity of interest (they are of the same magnitude) so that  one could use them in stopping criteria. Moreover,
in both cases,
$d_k$ is very close to $\widetilde{d}_k$ in the final convergence phase
that is suitable for stopping the algorithm.

In \Cref{fig:numexp_05_iphase} we demonstrate the technique of \Cref{sec:adaptd_init}
for the initial choice of $d_0$. As already observed in test
problems {\tt bcsttk04}, {\tt tmt\_sym}, and {\tt s3dkt3m2},
\Cref{alg:pseudo} can provide smaller values of $d_k$ than needed
in the initial quasi-stagnation phase.
%This can result in
%a visible underestimation of the $A$-norm of the error vector
%in the initial PCG iterations.
One can overcome this trouble by the strategy presented in \Cref{sec:adaptd_init} when approximating the nominator in
\eqref{eqn:initial} using the quantity \eqref{eqn:aUB}
that does not need any
a priory information
or using the Gauss--Radau upper bound, if available.
In our experiment we use \Cref{alg:pseudofull}, i.e.,
the quantity \eqref{eqn:aUB}.

For both test problems {\tt tmt\_sym} and {\tt s3dkt3m2},
the strategy based on the criterion~\eqref{eqn:initial}
determines properly a safe value $d_0$ that ensures
the prescribed relative accuracy of the initial estimate.
The value $d_0$ overestimates the ideal value $\widetilde{d}_0$
which means that the error has significantly decreased
in $d_0$ iterations of PCG, and that the initial estimate
is more accurate than required. In other words, using $d_0$
we have safely overcame the initial quasi-stagnation phase.

After $d_0$ is found in \Cref{alg:pseudofull},
we continue along the lines of \Cref{alg:pseudo}.
In particular, we immediately get the estimates
of the $A$-norm of the error vector for the initial plateau since
the values of $k$ are increased and the values of $d$ are
decreased on line 12 of \Cref{alg:pseudo}, without computing further PCG iterations. Once the criterion \eqref{eqn:lowerb_strategy}
is not met, we continue to use the standard strategy
of \Cref{alg:pseudo}. Again,
$d_k$ is very close to $\widetilde{d}_k$ in the
final convergence phase.

\begin{figure}
\centering
\begin{minipage}[c]{0.49\linewidth}
\centering
{\tt s3dkt3m2}+$d_0$\\[0.2cm]
\includegraphics[width=\linewidth]{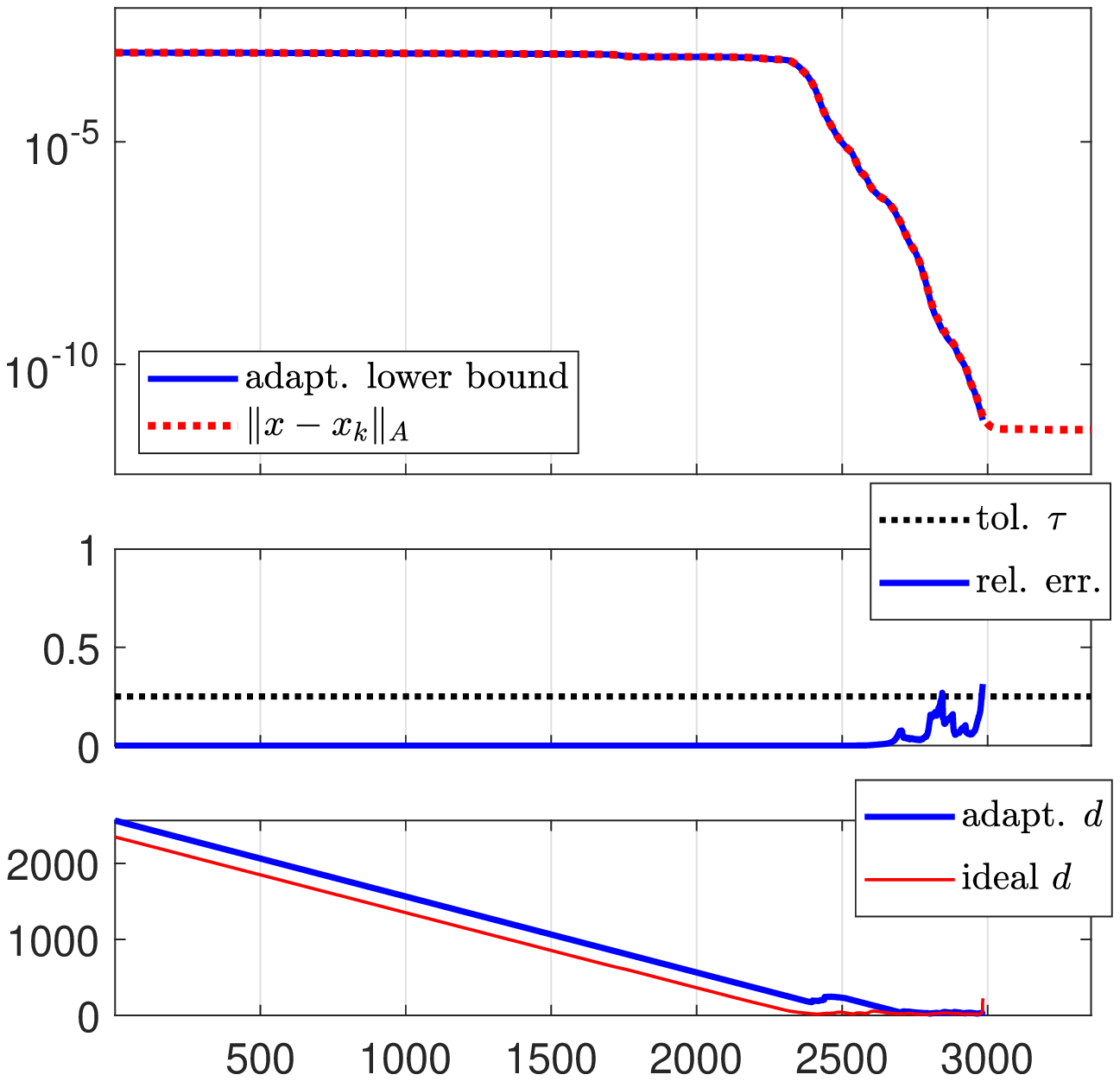}\\
PCG iterations
\end{minipage}
 \hfill
\begin{minipage}[c]{0.49\linewidth}
\centering
{\tt tmt\_sym}+$d_0$\\[0.2cm]
\includegraphics[width=\linewidth]{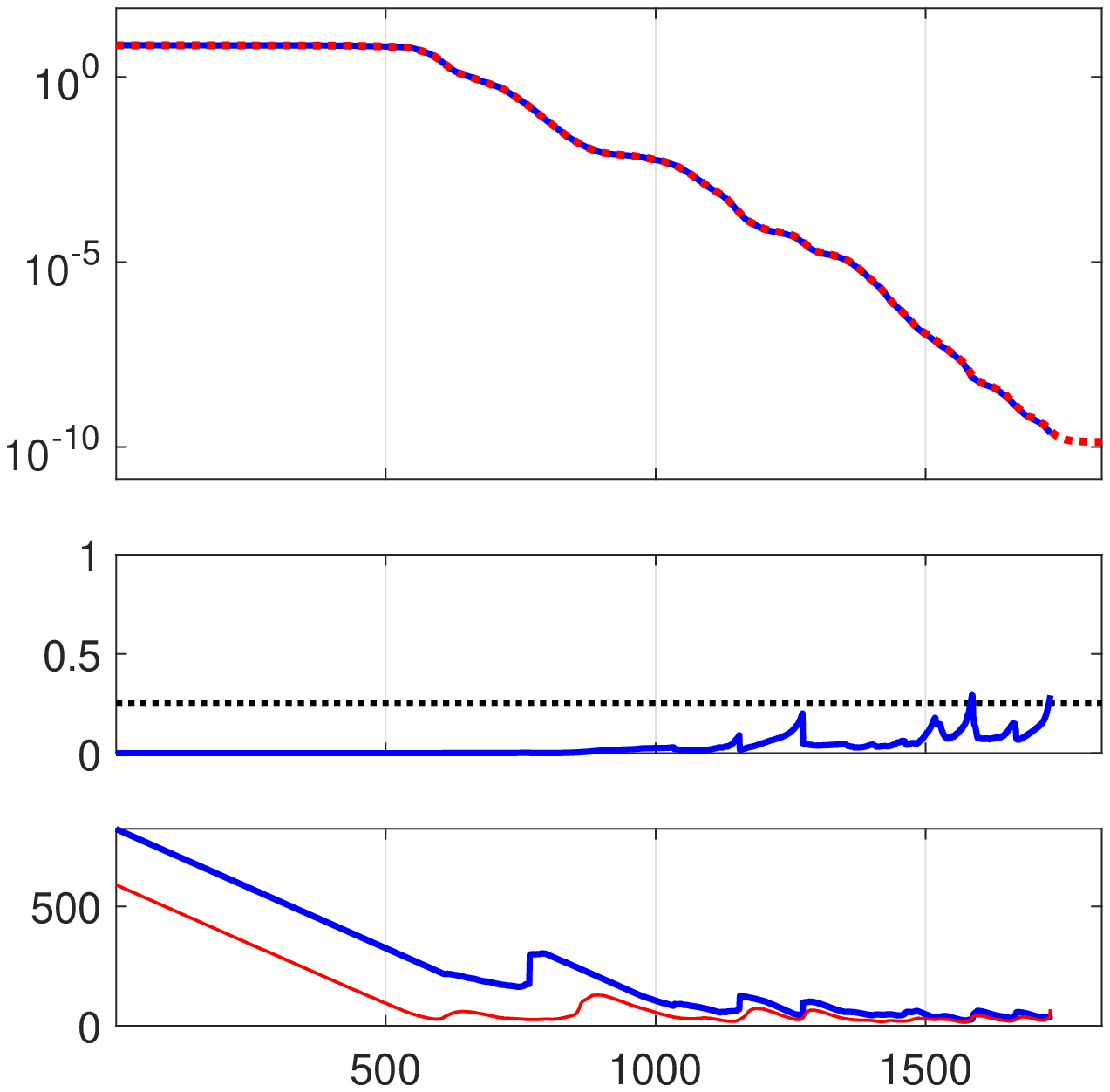}\\
PCG iterations
\end{minipage}

\caption{Matrices {\tt s3dkt3m2} and {\tt tmt\_sym} with the adaptive algorithm of \Cref{sec:adaptd_init} to determine the initial value~$d_0$: the $A$-norm of the error vector $x-x_k$ and the adaptive lower bound (top part), the relative error and the prescribed tolerance from~\eqref{eqn:tolerance} (middle part), the value $d_k$ and the ideal value~$\widetilde{d}_k$ (bottom part).} \label{fig:numexp_05_iphase}
\end{figure}

\begin{figure}
\centering
\begin{minipage}[c]{0.49\linewidth}
\centering
{\tt bcsstk04}\\[0.2cm]
\includegraphics[width=\linewidth]{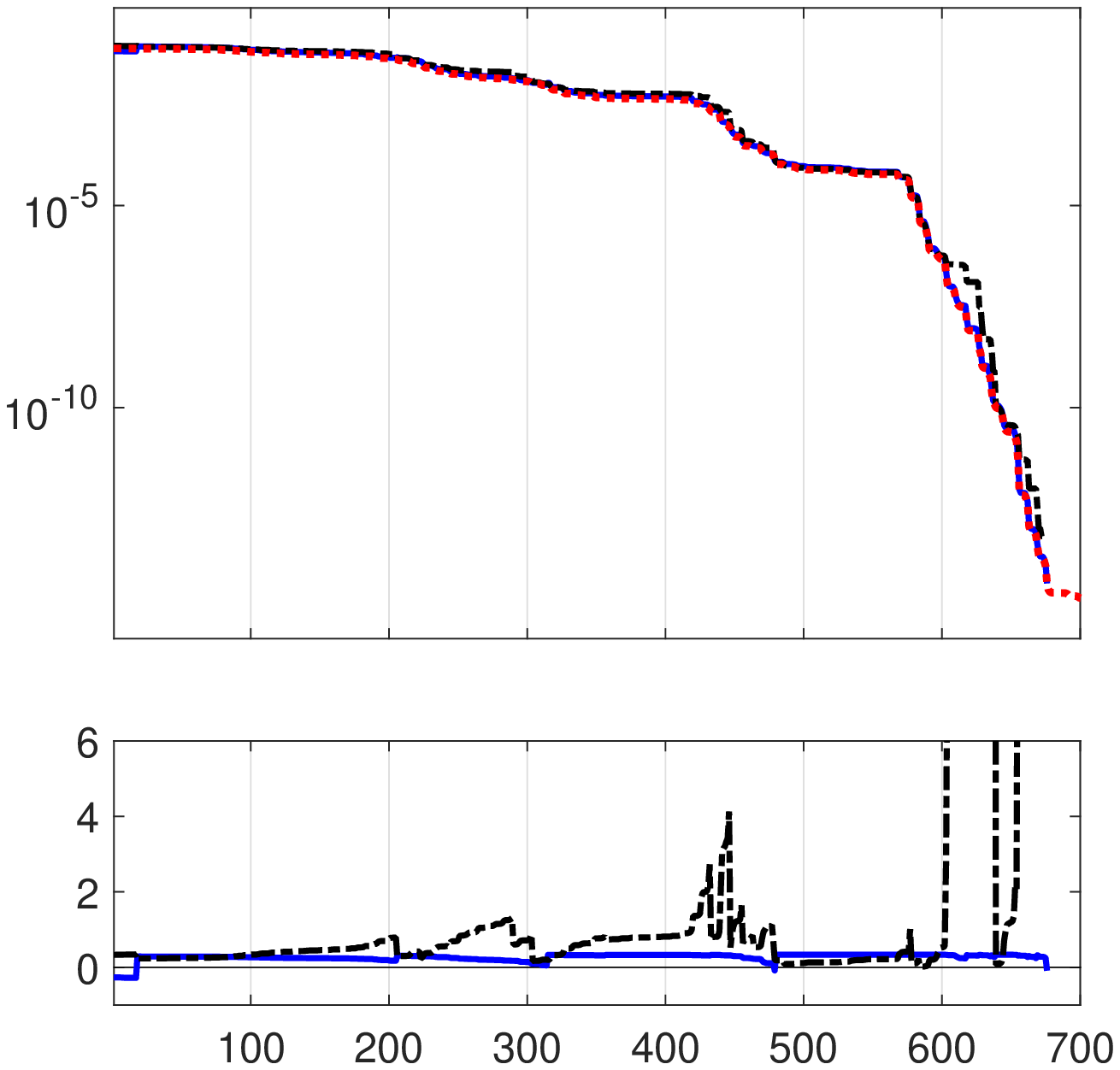}\\
PCG iterations
\end{minipage}
 \hfill
\begin{minipage}[c]{0.48\linewidth}
\centering
{\tt s3dkt3m2}\\[0.2cm]
\includegraphics[width=\linewidth]{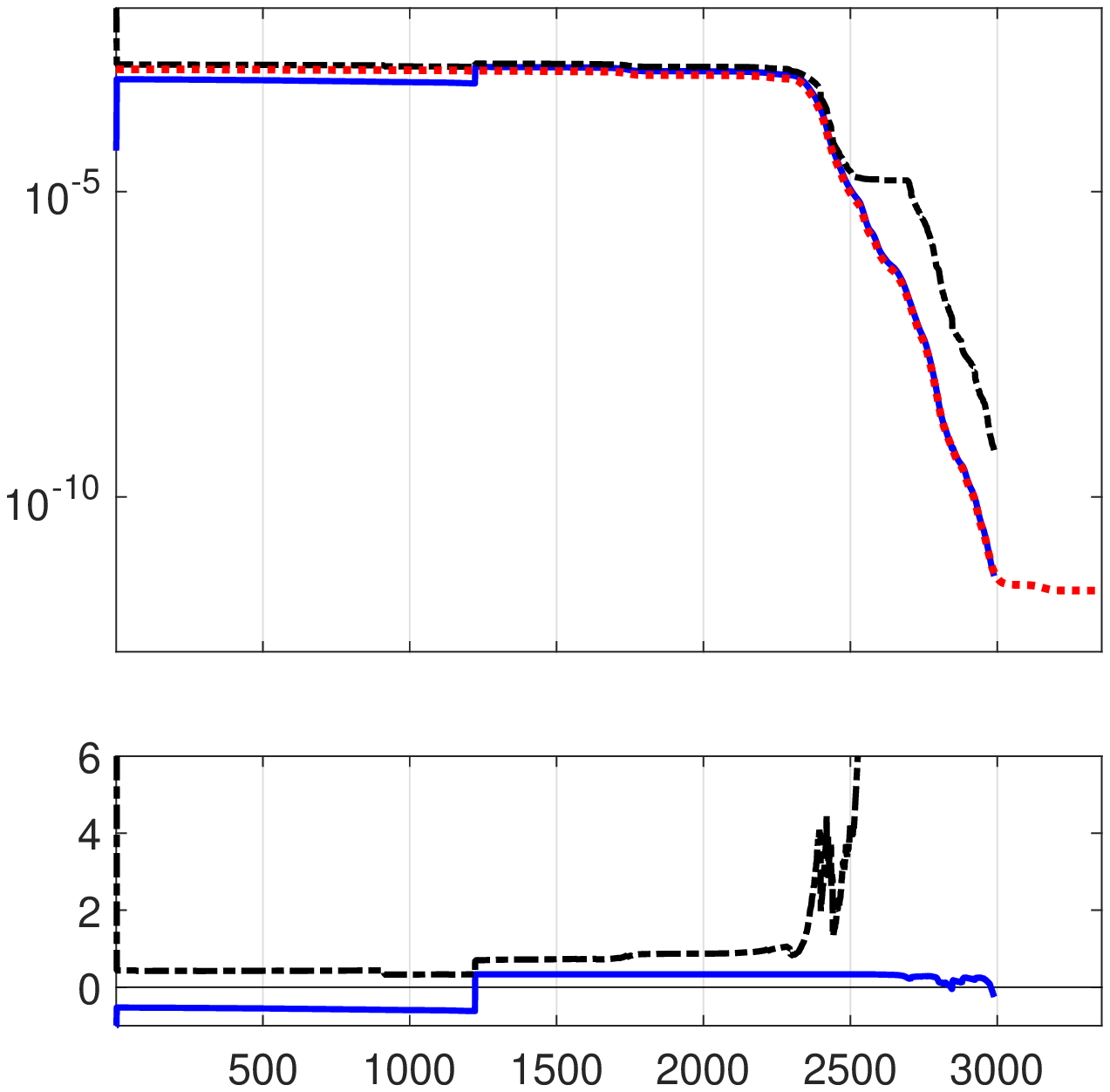}\\
PCG iterations
\end{minipage}

\caption{Matrices {\tt bcsstk04} and {\tt s3dkt3m2}. Top part: the $A$-norm of the error vector $x-x_k$ (red dashed curve), the square root of the heuristic upper bound \eqref{eqn:aUP} (blue solid curve), and the square root of the upper bound~\eqref{eqn:upperbound_delayGR} (black dash-dotted curve). Bottom part: the accuracy of bounds measured by the relative quantities~\eqref{eqn:qualityUpper}.}
%the quality of approximations using the upper or heuristic upper bounds; see \eqref{eqn:qualityUpper}.}
\label{fig:numexp_06_UB}
\end{figure}

Finally, in the last \Cref{fig:numexp_06_UB} we test the quality of the heuristic upper bound~\eqref{eqn:aUP},
\begin{equation}
 \label{eqn:ubound}
\frac{\cfgq_{k:k+d}}{1-\tau}
\end{equation}
that uses the adaptive choice of $d_k$ as in
\Cref{alg:pseudo}. We know that if $d_k$ is chosen such that
\eqref{eqn:tolerance} is satisfied (if $d_k\geq \widetilde{d}_k $),
then \eqref{eqn:ubound} represents an upper bound on~$\varepsilon_k$.
Note that even though the heuristic upper bound \eqref{eqn:ubound}
is not guaranteed to be an upper bound in general, its advantage is that no estimates of eigenvalues are needed.
We compare the heuristic upper bound  \eqref{eqn:ubound}
with the upper bound $\Omega_{k:k+d}^{\klein{(\mu)}}$ defined below
that uses the same number of $d_k$ terms as the quantity \eqref{eqn:ubound}.
In particular,
the bound $\Omega_{k:k+d}^{\klein{(\mu)}}$ is defined using
the identity \eqref{eqn:dlocaldecrease}
with the last term $\varepsilon_{k+d}$,
and the Gauss--Radau upper bound~\eqref{eq:GR}
on $\varepsilon_{k+d}$,
\begin{equation}
 \label{eqn:upperbound_delayGR}
 \| x-x_k \|^2_A \leq \Delta_{k:k+d-1} + \cfa_{k+d}^{\smu} \|r_{k+d}\|^2 \equiv \Omega_{k:k+d}^{\klein{(\mu)}}.
\end{equation}
The bound \eqref{eqn:upperbound_delayGR} depends on an a~priori given underestimate $\mu$ to the smallest eigenvalue of the (preconditioned) system matrix. In our experiments, 
the underestimate~$\mu$ was computed in the same way as described in
 \Cref{sec:GRupper}.

In the top part
of \Cref{fig:numexp_06_UB}
we plot
$\| x-x_k \|_A$ (red dots)
together with the upper bound $(\Omega_{k:k+d_k}^{\klein{(\mu)}})^{1/2}$
and the heuristic upper bound
given by the square root of \eqref{eqn:ubound}; in the bottom part,
we check the quality of bounds using the relative quantities
\begin{equation}
 \label{eqn:qualityUpper}
 \frac{\Omega_{k:k+d_k}^{\klein{(\mu)}}-\varepsilon_k}{\varepsilon_k}
 \ \mbox{(dashed)}
 \quad \mbox{ and } \quad  \frac{\frac{\Delta_{k:k+d_k}}{1-\tau}-\varepsilon_k}{\varepsilon_k}
  \ \mbox{(solid)}.
\end{equation}
Obviously, the considered heuristic upper bound is more accurate than
the Gauss--Radau upper bound in the convergence phase that is
crucial for stopping the iterations. 
In a future work we would like to explain the behavior of the Gauss--Radau upper bound (why it is delayed) and possibly fix the problem seen in the last iterations.

\section{Summary and concluding discussion}
 \label{sec:conclusions}

In this paper we focused on the accurate estimation of the error $\varepsilon_{k}$
in the P(CG) method. Our approach is based on the Hestenes and Stiefel
formula \eqref{eqn:dlocaldecrease} that is well preserved also
during finite precision computations. Our goal was
%We developed an efficient and
%cheap strategy that aims to reach a given relative accuracy of the lower
%bound $\Delta_{k:k+d}$ as an estimate of $\varepsilon_{k}$ via a
%convenient choice of the delay parameter $d$. In other words, the
%goal  is
to find a nonnegative integer $d$ (ideally
the smallest one) such that the relative accuracy of the estimate
$\Delta_{k:k+d}$ is less than or equal to a prescribed tolerance~$\tau$.
%If this is true, then the quantity $\Delta_{k:k+d}/(1-\tau)$
%represents an upper bound on $\varepsilon_{k}$.
The developed
technique for the adaptive choice of~$d$ is purely heuristic so that
the accuracy of
the estimate $\Delta_{k:k+d}$ cannot be guaranteed
in general.
Nevertheless, the heuristic strategy has shown to be robust and
reliable, which was demonstrated by numerical experiments.
Moreover, in the final stage of convergence that is suitable for stopping the iterations, the suggested value of $d$ is
usually very close to the optimal (ideal) one so that
one avoids unnecessary iterations.

One can also think of many other different adaptive strategies
that are based, e.g., on the Gauss--Radau
or anti-Gaussian quadrature rules, or on other techniques for approximating
the ratio \eqref{eqn:tolerance2}. We have tested many of them, and the strategy suggested in this paper seems to be efficient, simple, and robust.

Our techniques can also be adjusted for estimating the relative quantities
\[
 \frac{\|x-x_k\|_A^2}{\|x-x_0\|^2_A} \quad\mbox{or}\quad \frac{\|x-x_k\|_A^2}{\|x\|^2_A}
\]
using the fact that $\|x\|_A^2 = \|x-x_0\|^2_A +b^Tx_0 + r_0^T x_0$; see also \cite{StTi2005}.

We hope that the results presented in this paper will prove to be useful in practical computations.
They allow to approximate the error $\varepsilon_{k}$ at a negligible cost during (P)CG iterations without any user-defined parameter, while taking into account the prescribed relative accuracy of
the estimate.

\newpage
\noindent
\appendix
\section{\Cref{alg:pseudo} (MATLAB code, preconditioned version)}
\label{sec:MATLABcode}

\lstset{language=Matlab,%
    %basicstyle=\color{red},
    breaklines=true,%
    basicstyle=\small\ttfamily,
    morekeywords={matlab2tikz},
    keywordstyle=\color{blue},%
    morekeywords=[2]{1}, keywordstyle=[2]{\color{black}},
    identifierstyle=\color{black},%
    stringstyle=\color{mylilas},
%    commentstyle=\color{mygreen},%
    commentstyle=\color{green!40!black},
    showstringspaces=false,%without this there will be a symbol in the places where there is a space
    numbers=left,%
    numberstyle={\tiny \color{black}},% size of the numbers
    numbersep=9pt, % this defines how far the numbers are from the text
    %emph=[2]{word1,word2}, emphstyle=[2]{style},
   % morekeywords=[3]{findS,CG_with_adapt_est},
    deletekeywords={beta,sqrt,sum,find,max},
    deletekeywords={isempty,nargin,size,zeros,speye},
}

\vspace*{-3ex}
%{\scriptsize
\begin{lstlisting}
function [x,estim,delay] = pcga(A,b,tau,maxit,L,x)
%% ... preconditioned CG with the adaptive choice of d

r = b - A * x;                      % ... pcg initialization
z = L\r; z = L'\z; p = z;
rr = z' * r;
k = 1; d = 0;

for ell = 1:maxit+1

    RR    = rr;                     % ... begin cgiter(ell)
    Ap    = A * p;
    alpha = RR/(p' * Ap); 
    x     = x + alpha * p;
    r     = r - alpha * Ap;
    z     = L \ r; z = L' \ z;
    rr    = z' * r;
    beta  = rr / RR;             
    p     = z + beta * p;           % ... end cgiter(ell)
    
    Delta(ell) = alpha * RR;
    curve(ell) = 0; 
    curve = curve + Delta(ell);
    
    if ell > 1                      % ... adaptive choice of d
        S = findS(curve,Delta,k);
        
        num = S * Delta(ell);
        den = sum(Delta(k:ell-1));
        
        while (d >= 0) && (num/den <= tau)
            delay(k) = d;            
            estim(k) = den;
            k = k + 1; d = d - 1;
            den = sum(Delta(k:ell-1));
        end
        d = d + 1;
        %   ... use estim(k-1) in stopping criteria, if k > 1
    end
end
end % of function

function [S] = findS(curve,Delta,k) 
%% ... find the safety factor S using the tolerance 1e-4 

ind = find((curve(k)./curve) <= 1e-4, 1, 'last');
if isempty(ind), ind = 1; end

S = max(curve(ind:end-1)./Delta(ind:end-1));
end
\end{lstlisting}
%}

%\bibliographystyle{siamplain}
%\bibliography{./references}

\end{document}